\documentclass[11pt,reqno,letterpaper]{amsart}
\usepackage{color}
\usepackage[colorlinks=true, allcolors=blue,backref=page]{hyperref}
\usepackage{amsmath, amssymb, amsthm}
\usepackage{mathrsfs}
\usepackage{mathtools}
\usepackage{thmtools}
\usepackage{thm-restate}

\usepackage[noabbrev,capitalize,nameinlink]{cleveref}
\crefname{equation}{}{}
\usepackage{fullpage}
\usepackage[noadjust]{cite}
\usepackage{graphics}
\usepackage{pifont}
\usepackage{tikz}
\usepackage{bbm}
\usepackage[T1]{fontenc}

\DeclareMathOperator*{\argmin}{arg\,min}
\usetikzlibrary{arrows.meta}

\usepackage{environ}
\usepackage{framed}
\usepackage{url}
\usepackage[linesnumbered,ruled,vlined]{algorithm2e}
\usepackage[noend]{algpseudocode}
\usepackage[labelfont=bf]{caption}
\usepackage{cite}
\usepackage{framed}
\usepackage[framemethod=tikz]{mdframed}
\usepackage{appendix}
\usepackage{graphicx}
\usepackage[textsize=tiny]{todonotes}
\usepackage{tcolorbox}
\usepackage{enumerate}
\allowdisplaybreaks[1]
\usepackage{enumerate}
\usepackage{stmaryrd}
\usepackage[margin=1in]{geometry}

\usepackage[shortlabels]{enumitem}
\crefformat{enumi}{#2#1#3}
\crefrangeformat{enumi}{#3#1#4 to~#5#2#6}
\crefmultiformat{enumi}{#2#1#3}
{ and~#2#1#3}{, #2#1#3}{ and~#2#1#3}

\DeclareSymbolFont{symbolsC}{U}{pxsyc}{m}{n}
\SetSymbolFont{symbolsC}{bold}{U}{pxsyc}{bx}{n}
\DeclareFontSubstitution{U}{pxsyc}{m}{n}
\DeclareMathSymbol{\medcircle}{\mathbin}{symbolsC}{7}

\crefname{algocf}{Algorithm}{Algorithms}

\crefname{equation}{}{} 
\AtBeginEnvironment{appendices}{\crefalias{section}{appendix}} 

\usepackage[color,final]{showkeys} 

\colorlet{refkey}{orange!20}
\colorlet{labelkey}{blue!30}

\crefname{algocf}{Algorithm}{Algorithms}

\numberwithin{equation}{section}
\newtheorem{theorem}{Theorem}[section]
\newtheorem{proposition}[theorem]{Proposition}
\newtheorem{lemma}[theorem]{Lemma}

\crefname{claim}{Claim}{Claims}

\newtheorem{corollary}[theorem]{Corollary}

\newtheorem{remark}[theorem]{Remark}
\newtheorem*{question*}{Question}
\newtheorem{fact}[theorem]{Fact}

\theoremstyle{definition}
\newtheorem{definition}[theorem]{Definition}

\newtheorem*{definition*}{Definition}

\newcommand{\snorm}[1]{\lVert#1\rVert}
\newcommand{\sang}[1]{\langle #1 \rangle}

\newcommand{\mb}{\mathbb}

\newcommand{\mbm}{\mathbbm}
\newcommand{\mc}{\mathcal}
\newcommand{\mf}{\mathfrak}

\newcommand{\msf}{\mathsf}

\newcommand{\on}{\operatorname}

\newcommand{\eps}{\varepsilon}

\newcommand{\cS}{\ensuremath{\mathcal S}}

\let\originalleft\left
\let\originalright\right
\renewcommand{\left}{\mathopen{}\mathclose\bgroup\originalleft}
\renewcommand{\right}{\aftergroup\egroup\originalright}

\allowdisplaybreaks

\newif\ifpublic

\ifpublic

\newcommand{\ignore}[1]{}

\else

\fi

\newcommand{\bea}{\begin{eqnarray}}
\newcommand{\eea}{\end{eqnarray}}
\newcommand{\<}{\langle}
\renewcommand{\>}{\rangle}

\newcommand{\wt}{\widetilde}

\newcommand{\wh}{\widehat}

\def\ie{\text{i.e.~}}

\newcommand\eg{{\text{\eg~}}}

\def\oC{{\overline C}}
\def\eps{{\varepsilon}}

\def\supp{{\rm supp}}

\def\hx{\hat{x}}

\def\hq{{\widehat{q}}}

\def\bx{{\boldsymbol{x}}}

\def\cF{{\mathcal F}}

\def\cC{{\mathcal C}}

\def\sS{{\mathscr S}}

\def\vG{\vec G}
\def\vJ{\vec J}

\def\bsig{{\boldsymbol {\sigma}}}

\def\vlam{{\vec \lambda}}

\def\scale{{\mathsf{scale}}}

\def\bx{{\boldsymbol{x}}}

\def\Par{{\sf P}}
\def\de{{\rm d}}

\def\<{\langle}
\def\>{\rangle}

\def\cN{{\cal N}}

\def\b0{{\boldsymbol{0}}}

\DeclareMathOperator*{\plim}{p-lim}

\def\cI{{\mathcal I}}
\def\cS{{\mathcal S}}

\def\cB{{\mathcal B}}

\renewcommand{\b}{\mathbf{b}}

\def\fr{\frac}
\def\lt{\left}
\def\rt{\right}

\def\la{\langle}
\def\ra{\rangle}

\def\eps{\varepsilon}

\def\bbE{{\mathbb{E}}}

\def\bbP{{\mathbb{P}}}
\def\bbR{{\mathbb{R}}}

\def\bbZ{{\mathbb{Z}}}

\def\cB{{\mathcal{B}}}
\def\cF{{\mathcal{F}}}

\def\cN{{\mathcal{N}}}
\def\cP{{\mathcal{P}}}

\def\BY{{\mathsf{BY}}}
\def\LY{{\mathsf{LY}}}

\def\GS{{\mathrm{GS}}}

\newcommand{\oGS}{\overline{\GS}}

\def\sym{\mathsf{sym}}

\def\oM{\overline M}

\title{Free Energy Universality of Spherical Spin Glasses}

\author[A1]{Mehtaab Sawhney}
\address{Department of Mathematics, Columbia University, New York, NY 10027}
\email{m.sawhney@columbia.edu}

\author[A2]{Mark Sellke}
\address{Department of Statistics, Harvard University, 1 Oxford St, Cambridge, MA, 02138, USA}
\email{msellke@fas.harvard.edu}

\begin{document}

\maketitle
\begin{abstract}
We prove the free energy and ground state energy of spherical spin glasses are universal under the minimal moment assumptions. Previously such universality was known only for Ising spin glasses and random symmetric matrices, the latter being a celebrated result of Bai--Yin.
Our methods extend to $\ell^q$ balls for $q>2$, thus resolving a conjecture of Chen--Sen, and to tensor PCA.
\end{abstract}

\section{Introduction}

We prove the free energy of spherical mean-field spin glasses is universal in the disorder.
This simultaneously extends the easier Ising case \cite{Tal02,CH06,Cha05} and the Bai--Yin law \cite{BY88}.

\subsection{Formal statement of results}
The model consists of a random Hamiltonian defined by symmetric tensors with independent \emph{disorder} entries $\{J_{i_1,\ldots,i_p}\}_{1\leq i_1\leq \dots \leq i_p\leq N}$.
The entries have mean zero and equal variance, except for repeated indices $i_j=i_k$.
Given a sequence $(i_1,\ldots,i_p)\subset[N]^{p}$, considered as a multi-set, we define $|\{i_1,\ldots,i_p\}|$ to be the frequency counts and $|\{i_1,\ldots,i_p\}|!$ to denote the product of their factorials. Thus $\{1,1,1,2,3,3\}\subseteq [4]^{6}$ has $|\{1,1,1,2,3,3\}| = (3,1,2,0)$ and $|\{1,1,1,2,3,3\}|! = 3! \cdot 1! \cdot 2! \cdot 0!$.

\begin{definition}\label{def:dis-bound}
Fix $P\in\bbZ_+$. For each $1\leq p\leq P$ we consider independent $J_{i_1,\ldots,i_p}$ for $i_1\le i_2\le \ldots \le i_p$, with
\begin{equation}\label{eq:disorder}
\begin{aligned}
\mb{E}[J_{i_1,\ldots,i_p}] &= 0,\\
\mb{E}[J_{i_1,\ldots,i_p}^2] &= \frac{|\{i_1,\ldots,i_p\}|!}{p!}.
\end{aligned}
\end{equation}
We extend these values by permutation symmetry:
\[J_{i_1,\ldots,i_p} = J_{i_{\pi(1)},\ldots,i_{\pi(p)}},\quad\forall \pi\in \mf{S}_p.\]
For each $p$, the $p$-th disorder tensor of the model is $\vec J^{(p)}\in \bbR^{N^p}$, with $(i_1,\dots,i_p)$ entry $J_{i_1,\dots,i_p}$. 
We let $\vec J=(\vec J^{(1)},\dots,\vec J^{(P)})$ and say it obeys $(C,\eps)$-moment bounds if additionally
\[\mb{E}[|J_{i_1,\ldots,i_p}|^{2p+\eps}]\leq C,\quad\forall 1\leq p\leq P.\]
\end{definition}

Next define
\[
\cS_N=\{\bsig\in\bbR^N~:~\|\bsig\|_2=\sqrt{N}\},
\quad\quad\quad
\cB_N=\{\bsig\in\bbR^N~:~\|\bsig\|_2\leq \sqrt{N}\}
\]
with $\mu_N$ uniform measure on the former. We can now define the spherical spin glass model. 

\begin{definition}
Fix $P\in \bbZ_+$, constants $\gamma_1,\dots,\gamma_P\geq 0$, and disorder $\vJ$ as in \cref{def:dis-bound}. The spherical spin glass Hamiltonian is the random function $H_N:\cS_N\to\bbR^N$ given by
\begin{equation}
\label{eq:HN-def}
H_N(\bsig;\vJ)=H_N(\bsig) :=
    \sum_{p=1}^P
    \fr{\gamma_p}{N^{(p-1)/2}} \sum_{\substack{1\le \ell\le p\\ 1\le i_{\ell}\le N}}J_{i_1,\ldots,i_p} \sigma_{i_1}\cdots \sigma_{i_p}
=
\sum_{p=1}^P
H_N^{(p)}(\bsig).
\end{equation}

The associated partition function, free energy and ground state energy are:
\begin{align*}
Z_{\beta}(\vJ)&:=\int_{\bsig\in\cS_N}\exp(\beta H_N(\bsig))
    \de\bsig,\quad\forall \beta \in [0,\infty);
    \\
F_{\beta}(\vJ)&:=\frac{1}{\beta}\log Z_{\beta}(\vJ);
    \\
\GS(\vJ)&:=F_{\infty}(\vJ)= \sup_{\bsig\in\cS_N}H_N(\bsig).
\end{align*}
\end{definition}

We will frequently make use of the fact that $H_N$ has rotationally invariant covariance
\begin{align}\label{eq:gaussian-covariance}
\mb{E}[H_N(\bsig)H_N(\bsig')]=N\xi(\la\bsig,\bsig'\ra/N);
\quad\quad\quad
\xi(t):=\sum_{p=1}^P \gamma_p^2 t^p.
\end{align}
In the case where $J_{i_1,\ldots,i_p}$ is distributed as $\mc{N}\big(0,\frac{|\{i_1,\ldots,i_p\}|!}{p!}\big)$, it is easily checked from \eqref{eq:HN-def} that $H_N:\cS_N\to\bbR$ is the unique centered Gaussian process obeying \eqref{eq:gaussian-covariance}.
In particular, $H_N$ is rotationally invariant in law as a function on $\cS_N$. We denote Gaussian disorder by $G_{i_1,\ldots,i_p}$ or $\vec{G}$.

In the case of Ising spins (with $\cS_N$ replaced by $\{\pm 1\}^N$), Talagrand \cite{talagrand2006parisi} remarkably confirmed Parisi's prediction \cite{parisi1979infinite} for the free energy under the assumption that $\gamma_{k}=0$ for all odd $k$.
Soon after, Talagrand proved a similar formula for the spherical model \cite{Tal06}, which was predicted by \cite{crisanti1992spherical}. 
Later Panchenko \cite{panchenko2013parisi} gave a different proof removing the requirement of even interactions and establishing ultrametricity; it was extended to the spherical case by \cite{Che13}.
See also \cite{guerra2003broken,auffinger2015parisi,CS17,huang2023constructive}.

The (Crisanti--Sommers version of the) Parisi formula for the limiting free energy is as follows.
Let $x:[0,1] \to [0,1]$ be a right-continuous non-decreasing function such that $x(\hq) = 1$ for some $\hq < 1$ (which may depend on $x$) and set $\hx(q) = \int_q^1 x(q)~\de q$.
For fixed $\xi$, the Parisi functional is defined as
\[
\cP(x;\xi)
= \fr12 \lt\{
    \xi'(0) \hx(0)
    + \int_0^1 \xi''(q)\hx(q) ~\de q
    + \int_0^\hq \fr{\de q}{\hx(q)}
    + \log (1-\hq)
\rt\}.
\]

\begin{theorem}[{Parisi Formula, \cite{Tal06,Che13,CS17}}]
\label{thm:parisi}
For any $\beta\in [0,\infty)$ and Gaussian disorder,
\[\plim_{N\to\infty}F_{\beta}(\vG)/N = \lim_{N\to\infty}\mb{E}[F_{\beta}(\vG)/N]=\Par(\xi;\beta)\equiv \inf_x \cP(x;\beta^2\xi)/\beta
\]
where $x,\Par(\cdot)$ are as above.
Further, $\plim\limits_{N\to\infty} \GS(\vG)/N=\lim\limits_{N\to\infty} \bbE[\GS(\vG)/N]=\lim\limits_{\beta\to\infty} \Par(\xi;\beta)$.
\end{theorem}

In the Ising case, \cite{Tal02,CH06,Cha05} have shown universality of the free energy by a Lindeberg-style interpolation argument. 
It suffices to have independent entries with mean zero, unit variance, and bounded $2+\eps$ moments (or just $2$ moments with IID entries). 
See also \cite{auffinger2016universality,chen2024universality} for universality of other aspects of the model, and \cite{jagannath2022existence,chen2023some} for closer studies of heavy tailed disorder.
The spherical case is more challenging due to the presence of \emph{localized} $\bsig\in\cS_N$ with very large coordinates. 
Indeed, even in the $P=2$ case of random matrices at zero temperature, a celebrated result of Bai--Yin \cite{BY88} shows that $4$ moments are necessary and sufficient for universality of the maximum eigenvalue. 
Our first main result establishes free energy universality under minimal moment conditions.

\begin{theorem}\label{thm:main}
Fix $\eps\in (0,1)$ and $C\ge 1$. Consider disorder $J_{i_1,\ldots,i_p}$ obeying $(C,\eps)$-moment bounds. There exists $c\geq \Omega_P(\eps)$ such that for large $N$, uniformly in $\beta\in [0,\infty]$:
\begin{align*}
    \bbE[|F_{\beta}(\vJ)-\bbE F_{\beta}(\vG)|]\leq N^{1-c}.
\end{align*}
\end{theorem}

In Theorem~\ref{thm:main}, the disorder distributions may depend on $N$. 
However we also provide a variant \cref{thm:main} in the case that the $J_{i_1,\ldots,i_p}$ are IID (modulo repeated indices) and generated once and for all, rather than separately for each $N$.
In this case, directly generalizing \cite{BY88}, we show that $2p$ moments are necessary and sufficient for universality. Thus we in fact give a new proof of the Bai--Yin law avoiding any linear algebraic notions such as traces and eigenvalues.

\begin{theorem}\label{thm:bai-yin-var}
Fix probability distributions $\nu_1,\dots,\nu_P$ independent of $N$ with mean zero and variance $1$, such that each $\nu_p$ has finite $2p$-th moment. Let $J_{i_1,\ldots,i_p}\sim \Big(\frac{|\{i_1,\ldots,i_p\}|!}{p!}\Big)^{1/2}\nu_p$ be independent for $i_1\le \ldots \le i_p$ and $(i_1,\ldots,i_p)\in \mb{N}^p$, $J_{i_1,\ldots,i_p} = J_{i_{\pi(1)},\ldots,i_{\pi(p)}}$ for $\pi\in \mf{S}_p$, and which are fixed independently of $N$.
Then for any $\beta\in [0,\infty]$, we have the almost sure limit
\[
\lim_{N\to\infty}F_{\beta}(\vJ)/N = 
\Par(\xi,\beta).
\]
Conversely if any $\nu_p$ has infinite $2p$-th moment, then almost surely $\limsup\limits_{N\to\infty}\GS(\vJ)/N=\infty$.
\end{theorem}

Similar arguments also yield the following convergence in probability variant, paralleling \cite{lee2014necessary} in the matrix case (which showed much more detailed Tracy-Widom behavior at the edge under such a condition).
The proof is essentially identical to \cref{thm:bai-yin-var} and explained in \cref{rem:lee-yin}.

\begin{proposition}
\label{prop:lee-yin}
Fix probability distributions $\nu_1,\dots,\nu_P$ independent of $N$ with mean zero and variance $1$, such that for some $0<C<\infty$,
\begin{equation}
\label{eq:lee-yin}
\lim_{s\to\infty} s^{2p}\,\bbP^{x\sim \nu_p}[|x|\geq s]=0,\quad\forall 1\leq p\leq P.
\end{equation}
Then in the setting of \cref{thm:bai-yin-var} we have for all $\beta\in [0,\infty]$ the convergence in probability 
\[
\plim_{N\to\infty} F_{\beta}(\vJ)/N = 
\Par(\xi,\beta).
\]
Conversely if \eqref{eq:lee-yin} is false for some $p$, then $\limsup_{N\to\infty}\bbP[\GS(\vJ)\geq AN]>0$ for any $A>0$.
\end{proposition}

Subsequent to the main proofs, \cref{sec:extensions} provides extensions to tensor PCA, multi-species models, and $\ell^q$ balls for $q>2$. The last resolves \cite[Open Problem 5]{chen2023ell}.

Many other aspects of spin glasses are of interest besides the free energy. We plan to extend the methods here to prove universality for certain other behaviors in a future paper.

\subsection{Proof Outline}

As mentioned previously, the main difficulty compared to the Ising case is the presence of large ``localized'' coordinates, which prevent a Lindeberg-style interpolation from directly succeeding. 
In a preliminary step, we carefully truncate the disorder at a range of scales to reduce to the case that $\|\vJ^{(p)}\|_{\infty}$ is almost surely bounded by a slowly growing function of $N$. We remark here that the moment conditions in \cref{thm:main,thm:bai-yin-var} arise naturally and transparently from truncating the very largest entries. The precise moment cutoff arises precisely when these very large entries give rise to ``localized'' vectors with large value; we note that such an interpretation does not appear to fall out naturally from the approach of \cite{BY88} to the $P=2$ case.

Next we crucially observe that $\cS_N$ can be covered by a subexponential $e^{N^{1-c}}$ number of ``delocalized subspheres'', by essentially exhausting all possible localized coordinate patterns of points in $\cS_N$.
Each subsphere is centered at some $w\in\cB_N$ with all coordinates either large or zero. 
Further, all large coordinates of $w$ are essentially constant on its ``subsphere'', and other coordinates of $w$ are uniformly bounded. 
Having previously truncated the disorder, standard concentration results for Lipschitz functions of Gaussian and for convex Lipschitz functions of bounded independent variables apply to the free energy on each subsphere.
As a result, it suffices to show free energy universality on each subsphere individually.
The main point is that by design, Lindeberg-style interpolation applies on each subsphere, as all coordinates are either essentially constant or uniformly bounded.
In other words, the restriction to each delocalized subsphere behaves like an Ising spin glass.

Some additional arguments are required in the last stage. The main reason is that one cannot apply Lindeberg interpolation to the external field (linear term) of the Hamiltonian. Indeed in the Ising case, the free energy depends on the exact distribution of the external field and cannot be reduced to a finite number of moments.
This issue is further exacerbated on delocalized subspheres because the non-zero centering yields additional external field terms. (In particular these issues arise even when the original model has no external field, i.e. $\gamma_1=0$.) 
A recurring tool we use to address this difficulty is the rotational symmetry of both the sphere and Gaussian disorder, which lets us argue that the external field essentially only enters via its $L^2$ norm.

\subsection{Notation}
For $x\in \mb{R}^N$ and $p<\infty$, we let $\snorm{x}_{p}^p = \sum_{i=1}^{N}|x_i|^p$. Furthermore let $\snorm{x}_{\infty} = \sup_{1\le i\le n}|x_i|$. 
We write $f(N)\ll g(N)$ or $f(N)\leq O(g(N))$ or $g(N))\geq \Omega(f(N))$ if $f(N)\leq Cg(N)$ for $N$ sufficiently large, where $C$ may depend on $\max(\gamma_1,\dots,\gamma_P)$ but is otherwise an absolute constant.
For $1\leq p\leq P$ we will write $f(N)\ll_p g(N)$ or $f(N)\leq O_p(g(N))$ to indicate that $f(N)\leq C'g(N)$ for $N$ sufficiently large, where $C'$ may now depend on $p$.

For a closed subset $A\subseteq\bbR^N$, we write $\cP(A)$ for the space of Borel probability measures supported in $A$.
For $\bsig\in\bbR^N$, let $\supp(\bsig)=\{i\in [N]~:~\sigma_i\neq 0\}$.
We use $\mbm{1}$ to denote indicator functions for events.

We say a $p$-tensor $\vJ^{(p)}$ is symmetric if $J_{i_1,\dots,i_p}=J_{i_{\pi(1)},\dots,i_{\pi(p)}}$ for any permutation $\pi\in \mf{S}_p$ on $\{1,2,\dots,p\}$. Given vectors $x^1,\ldots,x^p\in \mb{R}^N$ and a $p$-tensor $\vJ^{(p)}$, we define 
\[
\sang{\vJ^{(p)}, \otimes_{j=1}^{p} x^j} 
= 
\sum_{ i_1,\dots,i_{p}=1}^N\vJ^{(p)}_{i_{1},\ldots,i_{p}}x^1_{i_1}\cdots x^p_{i_{p}}.
\]
We will write $\vJ^{(p)}\in \bbR^{N^p}$ for the $p$-tensor disorder and $\vJ$ for the entire disorder.

It will be important within the proof to consider measures on proper subsets of the sphere and their associated free energies. For general probability measures $\nu\in \cP(\cS_N)$, we define:
\begin{equation}\label{eq:free-energy-general-measure}
\begin{aligned}
Z_{\beta}(\vec J;\nu) &= \int_{\cS_N} \exp(\beta H_N(\bsig))\de\nu(\bsig),\\
F_{\beta}(\vec J;\nu) &= \frac{1}{\beta}\log\int_{\cS_N}\exp(\beta H_N(\bsig))\de\nu(\bsig).
\end{aligned}
\end{equation}
Finally to bound various error terms, we will occasionally use an absolute ground--state energy
\[
\oGS(\vJ):=\sup_{\bsig\in\cS_N}|H_N(\bsig)|=\max(\GS(\vJ),\GS(-\vJ)).
\]

\section{Truncation to bounded disorder}\label{sec:trunc}
In this section, we reduce our analysis to the case of almost uniformly bounded disorder. Throughout the proof we let $M = M_N$, $M_1 = M_{N,1}$, $\eta = \eta_{N}$, and $\delta = \delta_{N}$ be truncation parameters. We will assume throughout that 
\[
N^{-1/4}\le \eta\le \delta\le 1\le M,M_1\le N^{1/4}.
\]
We set $J\ge 1$ so $M\cdot 2^{J}=\eta \sqrt{N}$ (slightly adjusting constants so that $J\in\bbZ_+$). For $p\ge 2$, define 
\begin{align*}
J_{i_1,\dots,i_p}^{\msf{small}} &= J_{i_1,\dots,i_p}\cdot \mbm{1}_{|J_{i_1,\dots,i_p}|\leq M/2}
-\mb{E}[J_{i_1,\dots,i_p}\cdot \mbm{1}_{|J_{i_1,\dots,i_p}|\leq M/2}] \in [-M,M];\\
J_{i_1,\dots,i_p}^{\msf{scale}(j)}&=
J_{i_1,\dots,i_p}\cdot \mbm{1}_{|J_{i_1,\dots,i_p}|\in [M2^{j-1},M2^j)}
-\mb{E}[J_{i_1,\dots,i_p}\cdot \mbm{1}_{|J_{i_1,\dots,i_p}|\in [M2^{j-1},M2^j)},
\qquad\forall~0\leq j\leq J;\\
J_{i_1,\dots,i_p}^{\msf{large}}&= J_{i_1,\dots,i_p}\cdot \mbm{1}_{|J_{i_1,\dots,i_p}|\in[\eta\sqrt{N},\delta\sqrt{N}]}-\mb{E}[J_{i_1,\dots,i_p}\cdot \mbm{1}_{|J_{i_1,\dots,i_p}|\in [\eta\sqrt{N},\delta\sqrt{N}]}];\\
J_{i_1,\dots,i_p}^{\msf{tail}}&=
J_{i_1,\dots,i_p}\cdot \mbm{1}_{|J_{i_1,\dots,i_p}|> \delta\sqrt{N} }-\mb{E}[J_{i_1,\dots,i_p}\cdot \mbm{1}_{|J_{i_1,\dots,i_p}|> \delta\sqrt{N}}].
\end{align*}
For minor technical reasons, the truncation is slightly different when $p = 1$. In this case we define:
\begin{align*}
J_{i}^{\msf{small}} &= J_{i}\cdot \mbm{1}_{|J_{i}|\leq M_1/2}
-\mb{E}[J_{i}\cdot \mbm{1}_{|J_{i}|\leq M_1/2}] \in [-M_1,M_1];\\
J_{i}^{\msf{scale}(j)}&= 0,\qquad\forall~0\leq j\leq J; \\
J_{i}^{\msf{large}} &= 0\\
J_{i}^{\msf{tail}}&=
J_{i}\cdot \mbm{1}_{|J_{i}|>M_1/2}
-
\mb{E}[J_{i}\cdot \mbm{1}_{|J_{i}|>M_1/2}].
\end{align*}
We will handle the settings of Theorems~\ref{thm:main} and \ref{thm:bai-yin-var} in parallel below. 
In both cases, we will take $\eta_N=N^{-\frac{1}{4P}}$ and $M_N=N^{\frac{1}{4P}}$. 
In the former setting, the remaining values are also explicit:
\begin{equation}
\label{eq:delta-M1-eps}
\delta_N^{(\eps)}=N^{-\frac{\eps}{4P}},
\quad\quad
M_{N,1}^{(\eps)}=N^{\frac{1}{4P}}.
\end{equation}
However in the latter setting, the sequences $\delta_N^{\BY}\downarrow 0$ and $M_{N,1}^{\BY}\uparrow \infty$ will tend to their limits sufficiently slowly depending on the entry distributions $\nu_1,\dots,\nu_P$.
Thus we will use the superscripts $(\eps)$ and $\BY$ to distinguish the latter parameters.

The following standard estimates are ultimately where the sharp moment criteria enter.
\begin{lemma}\label{lem:Markov}
Fix $\alpha,\eps\in (0,1)$, $C\ge 1$. Let $\vec{J}$ obey $(C,\eps)$-moment bounds. Then 
\begin{equation}
\label{eq:Markov-1}
\mb{P}\bigg[
\sup_{1\le p\le P}
\sup_{1\leq i_1,\ldots,i_p\leq N}
|J_{i_1,\ldots,i_p}|
\ge 
\alpha \sqrt{N}
\bigg]
\ll_P 
C\alpha^{-2P- \eps}N^{-\eps}.
\end{equation}
Furthermore if $\mb{E}[|J_{i_1,\ldots,i_p}|^{2p}]\le C$ and $0\le q\le 2$, then
\begin{equation}
\label{eq:Markov-2}
\mb{E}[|J_{i_1,\dots,i_p}|^q\cdot \mbm{1}_{|J_{i_1,\dots,i_p}|>\alpha\sqrt{N}}]
\ll_p 
CN^{-p+\frac{q}{2}}\alpha^{-2p+2}.
\end{equation}
Finally if $\mb{E}[|J_{i}|^{2+\eps}]\le C$ then 
\begin{equation}
\label{eq:Markov-3}
\mb{E}[|J_{i}|^2\cdot \mbm{1}_{|J_{i}|>M_1/2}]
\ll_p 
C\eps^{-1}M_1^{-\eps}.
\end{equation}
\end{lemma}
\begin{proof}
We use Markov's inequality, and the tail-sum formula in the latter two cases:
\begin{align*}
\mb{P}\bigg[
\sup_{1\le p\le P}\sup_{i_1,\ldots,i_p}|J_{i_1,\ldots,i_p}|\ge t
\bigg]
&\ll_P \sum_{p = 1}^{P}Ct^{-(2p+\eps)} N^p
\ll_P 
C\alpha^{-2P- \eps}N^{-\eps};
\\
\mb{E}\big[
|J_{i_1,\dots,i_p}|^q\cdot \mbm{1}_{|J_{i_1,\dots,i_p}|> \alpha\sqrt{N}}
\big]
&\ll_p 
\int_{\alpha\sqrt{N}}^{\infty}
\frac{Cq}{t^{2p-q+1}}~\de t
\ll_p 
CN^{-p+\frac{q}{2}}\alpha^{-2p+2}
\\
\mb{E}[
|J_{i}|^2\cdot \mbm{1}_{|J_{i}|>M_1/2}
]
&\ll
\int_{M_1/2}^{\infty}\frac{C}{t^{1+\eps}}~\de t
\ll C\eps^{-1}M_1^{-\eps}.
\qedhere
\end{align*}
\end{proof}

\begin{lemma}\label{lem:diagonal}
Let $\vec{J}$ be as in \cref{thm:bai-yin-var}. There exists a decreasing sequence $\delta_N^{\BY}\downarrow 0$ and increasing sequence $M_{N,1}^{\BY}\uparrow \infty$ with $\delta_N^{\BY}\geq 1/\log N$ and $M_{N,1}^{\BY}\le \log N$ such almost surely:
\begin{align}
\label{eq:the-sum-which-is-finite}    
\sum_{N\ge 1}\mbm{1}\big[\sup_{1\le p\le P}\sup_{i_1,\ldots,i_p\in [N]}|J_{i_1,\ldots,i_p}|\ge \delta_N^{\BY} \sqrt{N}\big]&<\infty;
\\
\label{eq:external-field-BY-truncation} 
\sum_{N\ge 1}\mbm{1}\bigg[\sum_{i=1}^{N}|J_{i}|^2\cdot \mbm{1}[|J_{i}|\ge M_{N,1}^{\BY}/2]\ge \delta_N N\bigg]&<\infty.
\end{align}
\end{lemma}
\begin{proof}
For \eqref{eq:the-sum-which-is-finite} it suffices to handle the pure case where $\gamma_p=1$ and $\gamma_{p'}=0$ for all $p'\neq p$.
Indeed given suitable sequences $\big(\delta^{(p),\BY}_N\big)_{N\geq 1}$ for each pure case, taking $\delta_N^{\BY}=\max\limits_{1\leq p\leq P}\delta_N^{(p),\BY}$ then suffices in general. 
We set $N_j = 2^j$ and note that \eqref{eq:the-sum-which-is-finite} is finite if 
\[
\sum_{j\ge 1}\mbm{1}\big[\sup_{1\le p\le P}\sup_{i_1,\ldots,i_p\in [N_j]}|J_{i_1,\ldots,i_p}|\ge \delta_{N_j}^{\BY} \sqrt{N_j}/2\big]<\infty.
\]
As $\nu_p$ has bounded $2p$-th moment, we have for any fixed $\delta>0$:
\[\sum_{j\geq 1} N_j^{2pj} \mb{P}[|J_{i_1,\dots,i_p}|\geq \delta \sqrt{N_j}]<\infty.\] A standard diagonalization argument then proves that there exist $\delta_N^{\BY}\downarrow 0$ such that, as desired, 
\[
\sum_{j\geq 1}N_j^{2pj} \bbP[|J_{i_1,\dots,i_p}|\geq \delta_{N_j}^{\BY} \sqrt{N_j}]<\infty.
\]
The lower bound $\delta_N^{\BY}\geq 1/\log N$ is without loss of generality (and similarly for $M^{\BY}_{N,1}\leq \log N$).

For the second claim, fix a constant $m_1\ge 1$. By the strong law of large numbers, almost surely:
\[
\sum_{N\ge 1}\mbm{1}\bigg[\sum_{i=1}^{N}|J_{i}|^2\cdot \mbm{1}[|J_{i}|\ge m_1]
\ge
2\,\mb{E}\big[|J_{i}|^2\cdot \mbm{1}[|J_{i}|\ge m_1]\big]N\bigg]<\infty.
\]
Since $\lim\limits_{m_1\to\infty}\mb{E}\big[|J_{i}|^2\cdot \mbm{1}[|J_{i}|\ge m_1]\big]= 0$, a standard diagonalization argument completes the proof. 
\end{proof}

We take $\delta_{N}^{\BY}\downarrow 0$ and $M_{N,1}^{\BY}\uparrow \infty$ to be any sequences satisfying \cref{lem:diagonal}.
Recalling \eqref{eq:delta-M1-eps}, the parameters $(\delta,\eta,M,M_1)$ have now been defined in both cases.
For later use, we show the entries of $\vJ^{\msf{small}}$ have variance approximately $1$.

\begin{lemma}
\label{lem:approximate-conservation-of-variance}
    With $\delta\in \{\delta^{(\eps)},\delta^{\BY}\}$ as appropriate, there exists $\wt\delta\leq O_{P,\eps}(\delta)$ such that
    \begin{equation}
    \label{eq:approximate-conservation-of-variance-1}
    1-\wt\delta
    \leq
    \bbE[(J^{\msf{small}}_i)^2]
    \leq 1.
    \end{equation}
    Similarly for $p\geq 2$, and any $i_1,\dots,i_p\in [N]$:
    \begin{equation}
    \label{eq:approximate-conservation-of-variance-2}
    \frac{|\{i_1,\ldots,i_p\}|!}{p!}\cdot (1-\wt \delta)
    \le
    \bbE[(J^{\msf{small}}_{i_1,\dots,i_p})^2]
    \leq \frac{|\{i_1,\ldots,i_p\}|!}{p!}
    \end{equation}
\end{lemma}

\begin{proof}
    Firstly $\bbE[(J^{\msf{small}}_i)^2]$ is the variance of $J_i - (J_i\cdot \mbm{1}_{|J_i|\geq M_1/2})$, so
    \begin{equation}
    \label{eq:bias-variance-truncation}
    \bbE[(J^{\msf{small}}_i)^2]
    =
    1-\bbE[J_i^2\cdot \mbm{1}_{|J_i|\geq M_1/2}]
    -
    \bbE[(J_i\cdot \mbm{1}_{|J_i|\geq M_1/2})^2]
    \geq 
    1-2\,\bbE[J_i^2\cdot \mbm{1}_{|J_i|\geq M_1/2}].
    \end{equation}
    If $\vJ$ obeys $(C,\eps)$-moment bounds, then \eqref{eq:Markov-3} implies \eqref{eq:approximate-conservation-of-variance-1}, while in the other case the conclusion is implicit in \eqref{eq:external-field-BY-truncation}.
    
    We turn to proving \eqref{eq:approximate-conservation-of-variance-2} for $p\geq 2$.
    Noting that $\delta\geq M^{-2}=N^{-\frac{1}{2P}}$ in either case, \eqref{eq:Markov-2} yields
    \[
    \bbE[(J^{\msf{small}}_{i_1,\dots,i_p}\cdot \mbm{1}_{|J^{\msf{small}}_{i_1,\dots,i_p}|\geq M/2})^2]
    \leq 
    O_P(M^{-2})
    \leq 
    O_P(M^{-2p+2})
    \leq 
    O_P(\delta)
    .
    \]
    Arguing as in \eqref{eq:bias-variance-truncation} proves \eqref{eq:approximate-conservation-of-variance-2} as desired.
\end{proof}

We next show the contribution of $\vJ^{\msf{tail}}$ is uniformly negligible on $\cS_N$.

\begin{lemma}\label{lem:tail-rem}
Fix $\eps\in (0,1)$ and $C\ge 1$, and let $\delta^{(\eps)}$ and $M_1^{(\eps)}$ be as in \eqref{eq:delta-M1-eps}. If $\vec{J}$ obeys $(C,\eps)$-moment bounds, then
\[
\mb{P}[\oGS(\vec{J}^{\msf{tail}})\ge N^{1-\frac{\eps}{16P}}]\ll_{P,\eps}CN^{-\frac{\eps}{8P}}. 
\]
If $\vec{J}$ is instead as in \cref{thm:bai-yin-var}, we have the almost sure limit 
\[\lim_{N\to\infty}\oGS(\vec{J}^{\msf{tail}})/N = 0.\]
\end{lemma}
\begin{proof}
We first handle the $p=1$ disorder $\vJ^{(1)}$. In either case, 
\begin{align*}
\sup_{\boldsymbol{\sigma}\in \mc{S}_N}
\bigg|&\sum_{1\le i\le N}\sigma_i\cdot \big(J_{i_1,\ldots,i_p}\cdot \mbm{1}_{|J_{i_1,\dots,i_p}|> M_1}-\mb{E}[J_{i_1,\ldots,i_p}\cdot \mbm{1}_{|J_{i_1,\dots,i_p}|> M_1}]\big)\big|\\
&\ll \sqrt{N}\cdot \bigg(\sum_{i=1}^{N}|J_{i_1,\ldots,i_p}|^2\cdot \mbm{1}_{|J_{i_1,\dots,i_p}|> M_1} + \mb{E}[|J_{i_1,\ldots,i_p}|^2\cdot \mbm{1}_{|J_{i_1,\dots,i_p}|> M_1}]\bigg)^{1/2}.
\end{align*}
This is bounded by either \eqref{eq:Markov-3} in \cref{lem:Markov} (and a further application of Markov's inequality) in the case that $\vJ$ obeys $(C,\eps)$-moment bounds or \cref{lem:diagonal} in the setting of \cref{thm:bai-yin-var}.

We then handle the portion of the tensor with $p\ge 2$; we apply triangle inequality to merge these results. In the case of $(C,\eps)$-moment bounds, by applying \cref{lem:Markov}, it follows that all entries are bounded by $\delta^{(\eps)} \sqrt{N}$, except with probability $\ll_{P}CN^{-\frac{\eps}{4P}}$. Therefore it suffices to handle the terms which arise from $\mb{E}[J_{i_1,\ldots,i_p}\cdot \mbm{1}_{|J_{i_1,\dots,i_p}|> \delta^{(\eps)}\sqrt{N}}]$. From \eqref{eq:Markov-2} in \cref{lem:Markov}, we have 
\[
\mb{E}[J_{i_1,\ldots,i_p}\cdot \mbm{1}_{|J_{i_1,\dots,i_p}|> \delta^{(\eps)}\sqrt{N}}]\ll_{P}CN^{-p + \frac{1}{2}}N^{\frac{(2p-2)\eps}{4P}}\ll_{P}CN^{-p+1}.
\]
Therefore
\begin{align*}
\sup_{\boldsymbol{\sigma}\in \mc{S}_N}\bigg|
&\sum_{p=2}^{P}
\frac{\gamma_p}{N^{(p-1)/2}}
\sum_{\substack{1\le \ell\le p\\ 1\le i_{\ell}\le N}}
\mb{E}
\big[
J_{i_1,\ldots,i_p}
\cdot 
\mbm{1}_{|J_{i_1,\dots,i_p}|
>
\delta^{(\eps)}\sqrt{N}}
\big] 
\sigma_{i_1}\cdots\sigma_{i_p}\bigg|\\
&\ll_{P} \sup_{\boldsymbol{\sigma}\in \mc{S}_N}\sum_{p=2}^{P}\frac{\gamma_p}{N^{(p-1)/2}} \cdot  CN^{-p+1} \snorm{\sigma}_{1}^{p}\le N^{2/3}.
\end{align*}
In the final inequality we used that $N$ is sufficiently large. 

In the setting of \cref{thm:bai-yin-var}, \cref{lem:diagonal} implies that for all but finitely many $N$, it suffices to consider the contribution from $\mb{E}\big[J_{i_1,\ldots,i_p}\cdot \mbm{1}_{|J_{i_1,\dots,i_p}|> \delta^{\BY}\sqrt{N}}\big]$. These are handled exactly as in the previous case (since \eqref{eq:Markov-2} in \cref{lem:Markov} used only a $2p$-th moment bound).
\end{proof}

We have now removed the largest tensor entries. In the next step we handle entries which are slightly polynomially smaller.
The main observation is that it is unlikely for there to be multiple tensor entries of the relevant size which share any index $i\in [N]$. 
This step is important for the setting of \cref{thm:bai-yin-var} as the analysis of $J_{i_1,\ldots,i_p}^{\msf{scale}(j)}$ naturally loses logarithmic factors.

\begin{lemma}\label{lem:large-rem}
Fix $\eps\in (0,1)$ and $C\ge 1$. If $\vec{J}$ obeys $(C,\eps)$-moment bounds, then
\begin{equation}
\label{eq:large-rem-1}
\mb{P}[\oGS(\vec{J}^{\msf{large}})\ge N^{1-\frac{\eps}{5P}}]\ll_{P}CN^{-1/2}. 
\end{equation}
If instead $\vec{J}$ is as in \cref{thm:bai-yin-var}, we have the almost sure limit  
\begin{equation}
\label{eq:large-rem-2}
\lim_{N\to\infty}\oGS(\vec{J}^{\msf{large}})/N = 0.
\end{equation}
\end{lemma}
\begin{proof}
We assume $p\ge 2$ below, since the $p=1$ component of $\vec{J}^{\msf{large}}$ is zero by definition. 
Let $E_N(\eta/2)$ be the event that there exist $(i_1,\dots,i_p),(i_1',\dots,i_p')$ with some shared index $i_j=i_k'$ and
\[
    \min\big(|J_{i_1,\dots,i_p}|,|J_{i_1',\dots,i_p'}|\big)
    \geq 
    \eta\sqrt{N}/2.
\]
Its probability is upper bounded by
\[
\bbP[E_N(\eta/2)]
\leq 
\sum_{2\le p\le P}pN\cdot N^{2p-2} \cdot \bigg(\sup_{i_1,\ldots,i_p\in [N]}\frac{\mb{E}[|J_{i_1,\ldots,i_p}|^{2p}]}{(N^{1-\frac{1}{4P}}/2)^{2p}}\bigg)^2\ll_{P} \sup_{i_1,\ldots,i_p\in [N]}\mb{E}[|J_{i_1,\ldots,i_p}|^{2p}]N^{-1/2}.
\]
In the setting of \cref{thm:bai-yin-var}, $E_N(\eta/2)$ occurs for finitely many values $N=2^k$ by Borel--Cantelli, hence $E_N(\eta)$ occurs for finitely many $N\geq 1$.

We assume below that $E_N(\eta)$ does not hold. Then with $\delta\in \{\delta^{(\eps)},\delta^{\BY}\}$ as appropriate, we have deterministically:
\begin{align*}
\sup_{\boldsymbol{\sigma}\in \mc{S}_N}&\bigg|\sum_{p=2}^{P}\frac{\gamma_p}{N^{(p-1)/2}}\sum_{\substack{1\le \ell\le p\\ 1\le i_{\ell}\le N}}J_{i_1,\ldots,i_p}\cdot \mbm{1}_{|J_{i_1,\dots,i_p}|\in[\eta\sqrt{N},\delta\sqrt{N}]} \sigma_{i_1}\cdots\sigma_{i_p}\bigg|\\
&\ll_{P} \sup_{\boldsymbol{\sigma}\in \mc{S}_N}\sum_{p=2}^{P}\frac{\gamma_p}{N^{(p-1)/2}}\sum_{\substack{1\le \ell\le p\\ 1\le i_{\ell}\le N}}\delta\sqrt{N}\cdot \mbm{1}_{|J_{i_1,\dots,i_p}|\in[\eta\sqrt{N},\delta\sqrt{N}]}(|\sigma_{i_1}|^p + \cdots + |\sigma_{i_p}|^p)\\
&\ll_{P}\sup_{\boldsymbol{\sigma}\in \mc{S}_N}\sum_{p=2}^{P}\frac{\gamma_p}{N^{(p-1)/2}}\cdot \delta\sqrt{N} \snorm{\sigma}_p^{p}
\\
&\ll_P \delta N.
\end{align*}
The expectation terms in the definition of $\vec{J}^{\msf{large}}$ are bounded deterministically, using \eqref{eq:Markov-2} with $\eta=N^{-\frac{1}{4P}}$ in place of $\alpha$ therein:
\begin{align*}
\sup_{\boldsymbol{\sigma}\in \mc{S}_N}\bigg|\sum_{p=2}^{P}\frac{\gamma_p}{N^{(p-1)/2}}\sum_{\substack{1\le \ell\le p\\ 1\le i_{\ell}\le N}}\mb{E}[J_{i_1,\ldots,i_p}\cdot \mbm{1}_{|J_{i_1,\dots,i_p}|\in[\eta\sqrt{N},\delta\sqrt{N}]}] \sigma_{i_1}\cdots\sigma_{i_p}\bigg|
&\ll \sum_{p=2}^{P}|\gamma_p| N^p\cdot \frac{CN^{-p+\frac{1}{2}}N^{\frac{p-1}{2P}}}{N^{(p-1)/2}}
\\
&
\ll_P CN^{1/2}. 
\end{align*}
Combining directly yields \eqref{eq:large-rem-1} since $\delta N=N^{1-\frac{1}{4P}}$ is much smaller than $N^{1-\frac{1}{5P}}$. Since $\delta_N^{\BY}\downarrow 0$, \eqref{eq:large-rem-2} also follows.
\end{proof}

We now truncate moderate entries by multi-scale analysis, treating both settings identically. 
\begin{lemma}\label{lem:bound}
Fix $M = N^{\frac{1}{4P}}$ and $\eta = N^{-\frac{1}{4P}}$. Suppose that 
\[
\sup_{
\substack{1\le p\le P
\\
i_1,\ldots,i_p\in [N]}
}
\mb{E}[|J_{i_1,\ldots,i_p}|^{2p}]\le C.
\]
Then
\[\mb{P}\bigg[\sum_{0\le j\le J}\oGS(\vec{J}^{\msf{scale}(j)})\ge CN^{1-\frac{1}{5P}}\bigg]\ll_{P} N^{-\omega(1)}.\] 
\end{lemma}
\begin{proof}
By a union bound, it suffices to prove that for all $0\le j\le J$ and $2\le p\le P$, 
\[\mb{P}\bigg[\oGS(\vec{J}^{\msf{scale}(j),(p)})\ge CN^{1-\frac{1}{5P}}/(\log N)^{2}\bigg]\ll_P N^{-\omega(1)}.\] 
Here $\vec{J}^{\msf{scale}(j),(p)}$ denotes the $p$-tensor portion of $\vec{J}^{\msf{scale}(j)}$. We replace the law of $J^{\scale(j)}_{i_1,\dots,i_p}$ with a symmetrization $J^{\msf{sym}(j)}$ containing entries:
\[J^{\msf{sym}(j)}_{i_1,\dots,i_p}:= J^{\scale(j)}_{i_1,\dots,i_p}-\wt J^{\scale(j)}_{i_1,\dots,i_p}.\] 
Here $\wt J^{\scale(j)}$ is an independent copy of $J^{\scale(j)}$. By fixing $\bsig\in\cS_N$ achieving the value $\oGS(J^{\scale(j),(p)})$ and applying a truncated analog of \eqref{eq:gaussian-covariance} to $\wt J^{\scale(j)}$, we obtain
\begin{align*}
\mb{P}&\big[\oGS(\vJ^{\msf{sym}(j),(p)})\geq CN^{1-\frac{1}{5P}}/(2\log^2 N)~|~\oGS(\vec J^{\scale(j)})\geq CN^{1-\frac{1}{5P}}/(\log^2 N)\big]\geq 1/2.
\end{align*}
Note the entries of $J^{\msf{sym}(j)}$ are invariant in law under sign flips. We let $\eps^{(p)}_j=C(M2^j)^{-2p}$.
Since the expectation-corrections cancel in the definition of $\vJ^{\msf{sym}(j)}$, we find
\begin{equation}\label{eq:J-p-j-sparse}
\mb{P}[J^{\msf{sym}(j),(p)}_{i_1,\dots,i_p}\neq 0]\ll_{P}C(M2^j)^{-2p}.
\end{equation} 

Next, given $\bsig\in \cS_N$ and $K=10\log(N)\leq \log^2 N$, we can decompose its entries into dyadic scales via $\bsig=\sum_{k=0}^K \bsig^{(k)}$ where for $k<K$:
\[\bsig^{(k)}_i=\bsig_i\cdot \mbm{1}_{|\bsig_i|\in (2^{-k-1} \sqrt{N},2^{-k} \sqrt{N}]}.\] 
Therefore we can expand 
\begin{equation}
\label{eq:expand-p-power}
    \bsig^{\otimes p}=\sum_{k_1,\dots,k_p=0}^K
\bsig^{(k_1)}\otimes\dots\otimes \bsig^{(k_p)}.
\end{equation}
    As this consists of at most $(\log N)^{2p}$ terms, it suffices to fix $k_1,\dots,k_p$ and estimate
    \[
    \big\la 
    \bsig^{(k_1)}\otimes\dots\otimes \bsig^{(k_p)},
    \vJ^{\msf{sym}(j)}
    \big\ra.
    \]
    We restrict to the case $k_1\leq k_2\leq\dots\leq k_p$. 
    Let us fix $k_i$ and note that $|\on{supp}(\sigma^{(i)})|\leq 2^{2k_i}$.
    
    Recalling \eqref{eq:J-p-j-sparse}, Chernoff implies that with probability $1-N^{-\omega(1)}$, each $(p-1)$-tuple $(i_1,\dots,i_{p-1})$ is contained in at most $O_{P}\big(\eps^{(p)}_j N+\log^2 N\big)$ non-zero entries of $\vJ^{\msf{sym}(j),(p)}$.
    Call this event $E_j$.
    
    Let $\cF_j$ denote the sigma-algebra generated by the support of $\vJ^{\msf{sym}(j)}$; note that the signs of the entries of $\vJ^{\msf{sym}(j),(p)}$ are uniformly random conditioned on $\cF_j$. Let $\bsig_i$ denote any vector with at most $2^{2k_i}$ non-zero coordinates, each bounded in size by $2^{-k_i}\sqrt{N}$, and let $S_i = \on{supp}(\bsig_i)$.
    
    On the event $E_j$ we have
    \begin{align*}
    \bbE&\lt[\big\la 
    \bsig_1\otimes\dots\otimes \bsig_p,
   \vJ^{\msf{sym}(j),(p)}
    \big\ra^2~|~\cF_j\rt]\\
    &\qquad\qquad\ll_pN^{p/2}\cdot \sum_{\substack{1\le \ell\le p\\n_\ell\in S_\ell}}
    \frac{\mbm{1}_{(n_1,\dots,n_p)\in \on{supp}(\vJ^{\sym(j),(p)})}}{\prod_{i=1}^p 2^{2k_i}}
    \ll_pN^{p/2}\cdot \bigg(
    \frac{\eps^{(p)}_j N+\log^2 N}{2^{2k_p}}\bigg).
    \end{align*}
    
    Note that each non-zero entry in $\vJ^{\msf{sym}(j),(p)}$, conditional on $\cF_j$, is bounded by $M2^{j}$ and symmetric. Therefore $\sigma_i$ as before, and so $\big\la\bsig_1\otimes\dots\otimes \bsig_p,\vJ^{\scale(j)}\big\ra$ 
    is a sub--Gaussian random variable with variance proxy bounded by 
   \[\ll_P(M2^{j})^2 \cdot N^p\bigg(
    \frac{\eps^{(p)}_j N+\log^2 N}{2^{2k_p}}\bigg)\ll_P N^p\cdot \frac{CN(M2^{j})^{2-2p} + (M2^{j})^2 \log^2 N}{2^{2k_p}}.\]
    
    Let $\mc{N}_{k_i}$ be a $\frac{\sqrt{N}}{10p}$ net of vectors supported on at most $2^{2k_i}$ indices with all entries at most $2^{-k_i}\sqrt{N}$ in absolute value. Furthermore suppose that points in the net have supported on at most $2^{2k_i}$ indices with all entries at most $2^{-k_i}\sqrt{N}$ in absolute value; there exists $\mc{N}_{k_i}$ with size at most $\exp\big(O_p(2^{2k_i}(1+\log(2^{-2k_i}N)))\big)$.
    
    Furthermore, a straightforward argument gives
    \[\sup_{\bsig\in \mc{S}_N}\big|\big\la 
    \bsig^{(k_1)}\otimes\dots\otimes \bsig^{(k_p)},
   \vJ^{\msf{sym}(j),(p)}
    \big\ra\big|\le 2\sup_{\bsig_i\in \mc{N}_{k_i}}\big|\big\la 
    \bsig^{(k_1)}\otimes\dots\otimes \bsig^{(k_p)},
   \vJ^{\msf{sym}(j),(p)}
    \big\ra\big|.
    \]
    Thus union bounding over such nets, we have with probability $1-N^{-\omega(1)}$ (on the event $E_j$) that 
    \begin{align*}
    \sup_{\bsig_i\in \mc{N}_{k_i}}&\big|\big\la 
    \bsig^{(k_1)}\otimes\dots\otimes \bsig^{(k_p)},
   \vJ^{\msf{sym}(j),(p)}
    \big\ra\big|\\
    &\ll_P N^{p/2}\log N \cdot \sqrt{\log\prod_{i=1}^{p}|\mc{N}_{k_i}|}
   \cdot \bigg(\frac{CN(M2^{j})^{2-2p} + (M2^{j})^2 \log^2 N}{2^{2k_p}}\bigg)^{1/2}
    \\
    &\ll_{P} N^{p/2} (\log N)^3 \cdot 
 \max(C(M2^{j})^{1-p} N^{1/2}, M2^{j})\ll_P (\log N)^{3}\cdot N^{\frac{p+1}{2}-\frac{1}{4P}}.
    \end{align*}
Here we use that $M2^{j}\le \eta\sqrt{N}$ and $M2^{j}\ge M = N^{\frac{1}{4P}}$. Recalling \eqref{eq:expand-p-power} and summing over at most $(\log N)^{2P}$ possible sequences $(k_1,\dots,k_p)$ gives the same bound for $\la \bsig^{\otimes p},\vJ^{\msf{sym}(j),(p)}\ra = \la \bsig^{\otimes p},\vJ^{\msf{scale}(j),(p)}\ra$ up to these logarithmic factors.
Finally recalling that $\vJ^{(p)}$ is scaled by $N^{-(p-1)/2}$ in the definition \eqref{eq:HN-def} of $H_N(\bsig)$ completes the proof. 
\end{proof}

The remainder of the analysis will focus on $\vJ^{\msf{small}}$. 
Thanks to our truncations and \cref{lem:approximate-conservation-of-variance}, in the setting of either \cref{thm:main} or \cref{thm:bai-yin-var}, we have for $\wt\delta\leq O_{P,\eps}(\delta)$:
\begin{equation}
\label{eq:truncated-disorder-properties2}
\begin{aligned}
    \|\vJ^{\msf{small}}\|_{\infty}&\leq \max(M,M_1)\le N^{\frac{1}{4P}}\leq N^{1/64},
    \\
    \bbE[\vJ^{\msf{small}}_{i_1,\dots,i_p}]&=0,
    \\
     \frac{p!}{|\{i_1,\ldots,i_p\}|!}
     \bbE[(\vJ^{\msf{small}}_{i_1,\dots,i_p})^2]
    &\in
    \big[1-\wt\delta,1].
\end{aligned}
\end{equation}
(The last bound in the first line is without loss of generality as we may freely increase $P\to P'$ by setting $\gamma_{P+1}=\dots=\gamma_{P'}=0$.)

For technical reasons, we introduce standard Rademacher random variables $R_{i_1,\ldots,i_p}\in \{\pm 1\}$, which are independent of $\vJ$, and jointly independent of each other up to $\mf{S}_p$-symmetry (similarly to \cref{def:dis-bound}).
We choose constants $0\leq c_{i_1,\ldots,i_p}\leq O_P(\wt{\delta}^{1/2})$ such that
\[\frac{p!}{|\{i_1,\ldots,i_p\}|!}
     \bbE[(\vJ^{\msf{mod}})^2] = 1\]
where
\[\vJ^{\msf{mod}}_{i_1,\ldots,i_p} :=\vJ^{\msf{small}}_{i_1,\dots,i_p} + c_{i_1,\ldots,i_p}R_{i_1,\ldots,i_p}.\]
Therefore $\vJ^{\msf{mod}}$ satisfies:
\begin{equation}
\label{eq:truncated-disorder-properties}
\begin{aligned}
    \|\vJ^{\msf{mod}}\|_{\infty}&\leq \oM:=\max(M,M_1) + 1\le 2N^{\frac{1}{4P}}\leq N^{1/64},
    \\
    \bbE[\vJ^{\msf{mod}}_{i_1,\dots,i_p}]&=0,
    \\
     \frac{p!}{|\{i_1,\ldots,i_p\}|!}
     \bbE[(\vJ^{\msf{mod}}_{i_1,\dots,i_p})^2]
    &=1.
\end{aligned}
\end{equation}

The following elementary bound handles the difference between $\vJ^{\msf{mod}}$ and $\vJ^{\msf{small}}$.
\begin{lemma}\label{lem:diff-bound}
Let $\vJ^{\msf{mod}}$ and $\vJ^{\msf{small}}$ be as above. Then for $\beta\in [0,\infty]$ we have that 
\[\big|\mb{E}[F_{\beta}(\vJ^{\msf{mod}})] - \mb{E}[F_{\beta}(\vJ^{\msf{small}})]\big|\ll_P \wt{\delta}^{1/2} N.\]
\end{lemma}
\begin{proof}
We have that 
\[\big|\mb{E}[F_{\beta}(\vJ^{\msf{mod}})] - \mb{E}[F_{\beta}(\vJ^{\msf{mod}})]\big|\le \mb{E}[\oGS(\vJ^{\msf{mod}} - \vJ^{\msf{small}})].\]
The entries of $\vJ^{\msf{mod}} - \vJ^{\msf{small}}$ are mean zero and bounded by $O_{P}(\wt\delta^{1/2})$. One may couple the coordinates of $(\vJ^{\msf{mod}} - \vJ^{\msf{small}})_{i_1,\ldots,i_p}$ with independent (up to symmetry) Gaussians $\wt{G}_{i_1,\ldots,i_p}$ such that $\mb{E}[\wt{G}_{i_1,\ldots,i_p}|\vJ^{\msf{mod}} - \vJ^{\msf{small}}] = (\vJ^{\msf{mod}} - \vJ^{\msf{small}})_{i_1,\ldots,i_p}$. Then by Jensen's inequality, 
\[\mb{E}[\oGS(\vJ^{\msf{mod}} - \vJ^{\msf{small}})]\le \mb{E}[\oGS(\vec{\wt{G}})].\]
As the entries of $\wt{G}$ are Gaussians of variance bounded by $O_P(\wt{\delta})$, Slepian's lemma implies
\[
\mb{E}[\oGS(\vec{\wt{G}})]\ll_{P} \wt{\delta}^{1/2} \cdot \mb{E}[\oGS(\vec{\wt{G}})]\ll_{P} \wt{\delta}^{1/2} N.
\]
In the final line we have used $\mb{E}[\oGS(\vec{\wt{G}})]\ll_{P} N$ which follows by e.g. using Slepian to compare to the setting of \cref{thm:parisi} (or by a direct Dudley entropy integral). 
\end{proof}

In \cref{sec:decomp} to \cref{sec:bound-dis}, we study disorder $\vJ$ which is assumed to obey the properties \eqref{eq:truncated-disorder-properties}.
In \cref{sec:finish} we will apply these results to $\vJ^{\msf{mod}}$ to conclude.

\section{Decomposition of the sphere}\label{sec:decomp}
We next decompose the uniform measure $\mu_N$ on $\cS_N$ into subexponentially many smaller measures. On the support of each smaller measures, the large coordinates will be essentially constant. 

Fix a parameter $T\in [\log N,N^{1/4}]$ (we will eventually take $T$ to be a small power of $N$). We first show that nets controlling only the large coordinates of $x\in \mc{S}_N$ have subexponential size. 
\begin{fact}\label{fct:small-net}
Fix $\log N\le T\le \sqrt{N}$ and let 
\[
\mc{L}_T = \big((-\infty,-T]\cup \{0\}\cup [T,\infty)\big)^N \cap 
\cB_N.
\]
There exists a constant $C = C_{\ref{fct:small-net}}>0$ such that for any $\iota>0$, $\mc{L}_T$ admits a $\iota \sqrt{N}$ net $\mc{N}_{T,\iota}\subseteq \mc{L}_T$ of size $|\mc{N}_{T,\iota}|\leq (CN/\iota)^{N/T^2}$.
\end{fact}
\begin{proof}
It is standard that the unit sphere in $L$-dimensions has a $\iota$-net of size $(C/\iota)^{L}$. Note that any point in $\cS_N$ has at most $N/T^2$ coordinates larger than $T$. Unioning over all subsets of $[N]$ with cardinality at most $N/T^2$ gives the desired result since $\binom{N}{N/T^2}\le N^{N/T^2}$.
\end{proof}

For the remainder of the proof, we fix $\iota = N^{-16P}$ and $\mc{N}_{T,\iota}$ from \cref{fct:small-net}. Given $x\in \mc{S}_N$, define:
\begin{align*}
x_{T} &= \sum_{i\in [N]}\mbm{1}_{|x_i|\ge T}x_ie_i.
\end{align*}
By construction, $x_T\in \mc{L}_T$. 
For each $w\in\mc{N}_{T,\iota}$, let
\begin{align*}
\mc{C}_w = \{x\in \mc{S}_N: w \in \argmin_{w'\in \mc{N}_{T,\iota}}\snorm{x_T-w'}_2\}.
\end{align*}
Then clearly $\mc{S}_N = \bigcup_{w\in \mc{N}_{T,\iota}}\mc{C}_w$.
We also define 
\[x_w = \sqrt{N-\snorm{w}_2^2}\cdot\frac{x-x_T}{\snorm{x-x_{T}}_2}.
\]
If $x\in \mc{C}_w$ is uniformly random, we let $\mu_w$ be the law of $x_w$.
We let $\mc{S}_w$ be the sphere of radius $\sqrt{N-\snorm{w}_2^2}$ on the coordinates $[N]\setminus \on{supp}(w)$, with all zero entries in $\supp(w)$.
Let $\mu_w^{\ast}$ be the uniform probability measure on $\mc{S}_w$.

\begin{lemma}\label{lem:rnd-point}
Fix $\log N\le T\le N^{1/4}$ and let $x\in \cS_N$ and $w\in\cN_{T,\iota}$.
Then the following hold:
\begin{enumerate}[label=(\alph*)]
    \item 
    \label{it:rnd-point-1}
    $\snorm{x_w + w}_2 = \sqrt{N}$
    \item 
    \label{it:rnd-point-2}
    $\snorm{x_w + w - x}_2\le 4\iota^{1/4} N$
    \item 
    \label{it:rnd-point-3}
    $\snorm{x_w}_{\infty}\le 2T$
    \item 
    \label{it:rnd-point-4}
    For all $y\in \mc{S}_w$, we have
    $\frac{d\mu_w}{d\mu^{\ast}_w}(y)\le 1 + N^{-\omega(1)}$.
\end{enumerate}
\end{lemma}
\begin{proof}
Point~\ref{it:rnd-point-1} is immediate because $\snorm{x_w}_2 = \sqrt{N-\snorm{w}_2^2}$ and $\on{supp}(x_w)\cap \on{supp}(w) = \emptyset$. 

For \ref{it:rnd-point-2}, note that 
\begin{align*}
\snorm{x_w + w - x}_2 &= \snorm{x_w + w - x_{T}-(x-x_T)}_2\le \snorm{x_w -(x-x_{T})}_2 + \snorm{w - x_{T}}_2\\
&\le 2\iota\sqrt{N} + \bigg|\sqrt{N-\snorm{w}_2^2}-\snorm{(x-x_{T})}_2\bigg|\,.
\end{align*}
If $\snorm{x-x_T}_2\le \iota^{1/2} \sqrt{N}$, then $\snorm{w}_2 \ge \snorm{x_T}_2 - \iota\sqrt{N}\ge (1-2\iota^{1/2})\sqrt{N}$. In this case, 
\[\Big|\sqrt{N-\snorm{w}_2^2}-\snorm{(x-x_{T})}_2\Big|\le (2\iota^{1/2} N)^{1/2} + \iota^{1/2} \sqrt{N}\le 2\iota^{1/4}\sqrt{N}.\]
Else $\snorm{x-x_T}_2\ge \iota^{1/2} \sqrt{N}$ and so
\begin{align*}
\Big|\sqrt{N-\snorm{w}_2^2}-\snorm{(x-x_{T})}_2\Big| &= \frac{(N-\snorm{w}_2^2) - \snorm{(x-x_{T})}_2^2}{\sqrt{N-\snorm{w}_2^2} + \snorm{(x-x_{T})}_2}\\
&\le \frac{\snorm{w}_2^2 - \snorm{x_T}_2^2}{\snorm{(x-x_{T})}_2}\le \frac{2\sqrt{N}(\iota\sqrt{N})}{\iota^{1/2}\sqrt{N}}\le 2\iota^{1/2}\sqrt{N}.
\end{align*}

For \ref{it:rnd-point-3}, note that 
\[\snorm{x_w}_\infty \le \min\bigg(\frac{\sqrt{N-\snorm{w}_2^2} \, T}{\snorm{x-x_T}_2}, \sqrt{N-\snorm{w}_2^2}\bigg).\]
Now $\snorm{x-x_T}_2 = \sqrt{N - \snorm{x_T}_2^2}$. If $\snorm{w}_2\ge \sqrt{N-1}$, the second term implies $\snorm{x_w}_{\infty}\le 1$ as desired. Otherwise $\big|\snorm{w}_2-\snorm{x_T}_2\big|\le \iota\sqrt{N}$ and thus $\frac{\sqrt{N-\snorm{w}_2^2}}{\snorm{x-x_T}_2}\le 2$. This implies \ref{it:rnd-point-3}. 

For \ref{it:rnd-point-4} we fixed $w$ and condition on $x_T$. 
This determines whether $x\in \mc{C}_w$, and we prove the result with $\mu_w$ replaced by the conditional law of $x_w$. This suffices by averaging over $x_T$ at the end. 
Under the conditioning, $(x-x_T)$ is uniformly random on the sphere of radius $\sqrt{N-\sqrt{\|x_T\|_2}}$ in coordinates $[N]\setminus \on{supp}(w)$, except that $\snorm{x-x_T}_{\infty}\le T$. Since $T\ge \log N$, we find that the portion of the sphere of radius $\sqrt{N-\sqrt{\|x_T\|_2}}$ with $\snorm{\cdot}_{\infty}\le \log N$ is an $1-N^{-\omega(1)}$ fraction of the sphere. Then \ref{it:rnd-point-4} follows because $x_w$ is a rescaling of $(x-x_T)$ by a factor depending only on $x_T$. 
\end{proof}

By construction $\mu_N$ is approximated by a convex combination of $w+\mu_w$, the pushforward of $\mu_w$ under $x\mapsto w+x$.
Using \cref{lem:rnd-point}, we transfer this approximation to free energies.

\begin{lemma}\label{lem:meas-decomp}
For any $\vJ=\big(\vJ^{(1)},\dots,\vJ^{(P)}\big)$,
\[\bigg|F_{\beta}(\vec{J})-\frac{1}{\beta}\log\Big(\sum_{w\in \mc{N}_{T,\iota}}\mu_N(\cC_w) \cdot Z_{\beta}(\vec{J},w + \mu_w)\Big)\bigg|\le \snorm{J}_{\infty}\cdot N^{-2P}.\]
\end{lemma}
\begin{proof}
For $x,y\in \mb{R}^N$, we have 
\[\snorm{x^{\otimes p} - y^{\otimes p}}_2\le p\cdot \max(\snorm{x}_2,\snorm{y}_2)^{p-1}\snorm{x-y}_2\le p \cdot (\snorm{x}_2 + \snorm{y}_2)^{p-1}\cdot \snorm{x-y}_2.\]
With $\mb{E}_{\bsig\in \mc{C}_w}$ indicating that $\bsig\in \mc{C}_w$ is uniformly random, $F_{\beta}(\vec{J})$ is approximated by:
\begin{align*}
F_{\beta}(\vec{J}) &= \frac{1}{\beta}\log\Big(\sum_{w\in \mc{N}_{T,\iota}}\mu_N(\cC_w) \cdot \mb{E}_{\bsig\in \mc{C}_w}\big[\exp\big(\beta\sang{\vec{J},\bsig^{\otimes p}})\big]\Big)\\
&= \frac{1}{\beta}\log\Big(\sum_{w\in \mc{N}_{T,\iota}}\mu_N(\cC_w) \cdot \mb{E}_{\bsig\in \mc{C}_w}\big[\exp\big(\beta\sang{\vec{J},(\bsig_T + (\bsig-\bsig_T))^{\otimes p}})\big]\Big)\\
&= \frac{1}{\beta}\log\Big(\sum_{w\in \mc{N}_{T,\iota}}\mu_N(\cC_w) \cdot \mb{E}_{\bsig\in \mc{C}_w}\big[\exp\big(\beta\sang{\vec{J},(w + (\bsig-\bsig_T))^{\otimes p}})\big]\Big) \pm \snorm{\vec{J}}_{\infty}\cdot N^{-3P}\\
&= \frac{1}{\beta}\log\Big(\sum_{w\in \mc{N}_{T,\iota}}\mu_N(\cC_w) \cdot \mb{E}_{\bsig\in \mc{C}_w}\big[\exp\big(\beta\sang{\vec{J},(w + \bsig_w)^{\otimes p}})\big]\Big) \pm 2\snorm{\vec{J}}_{\infty}\cdot N^{-3P}\\
&= \frac{1}{\beta}\log\Big(\sum_{w\in \mc{N}_{T,\iota}}\mu_N(\cC_w) \cdot Z_{\beta}(\vec{J},w + \mu_w)\Big) \pm \snorm{\vec{J}}_{\infty}\cdot N^{-2P}. \qedhere
\end{align*}
\end{proof}

\section{Lindberg Exchange and Concentration for Bounded Disorder}\label{sec:add-input}
In this section, we provide the necessary Lindeberg exchange argument and concentration of free energy for bounded disorder. As previously mentioned, the former is the key ingredient for the Ising case \cite{Tal02,CH06,Cha05}. The latter is a consequence of Talagrand's concentration inequality for convex functions of bounded independent variables (and of Gaussian isoperimetry in the Gaussian case).

\subsection{Lindeberg exchange argument}\label{sec:exch}
In this we include a Lindberg exchange argument which allows us to deduce universality of free energy under appropriate boundedness constraints on the entries. Such computations appear essentially in works of Talagrand \cite{Tal02} and Carmona and Hu \cite[Section~2]{CH06} but are presented as interpolations. We phrase these arguments in the flavor of the more traditional Lindeberg exchange method, following Chatterjee \cite{Cha05}.

Given sequences $a = (a_1,\ldots,a_n)\in \mb{R}^{n}$ and $b = (b_1,\ldots,b_n)\in \mb{R}^n$, and $\beta\in \mb{R}$, we may define the free energy
\begin{align*}
F_{a,b,\beta}(\xi) &= \frac{1}{\beta}\log Z_{a,b,\beta}(\xi),
\\
Z_{a,b,\beta}(\xi) &= \sum_{i\in [n]}\exp\big(\beta (a_i \xi + b_i)\big).
\end{align*}
Note that 
\[\bigg|\frac{d}{d\xi}F_{a,b,\beta}(\xi)\bigg| = \bigg|\frac{\sum_{i\in [n]}a_i\exp\big(\beta (a_i \xi + b_i)\big)}{Z_{a,b,\beta}(\xi)}\bigg|\le \snorm{a}_{\infty}.\]
Furthermore H\"older's inequality implies that $F_{a,b,\beta}(\xi)$ is convex. We now compute the second and third derivatives of $F_{\alpha,\beta}(\xi)$. We have that 
\begin{align*}
\frac{d^2}{d\xi^2}F_{a,b,\beta}(\xi) &= \frac{(\sum_{i\in [n]}a_i\exp(\beta a_i\xi + b_i))\sum_{i\in [n]} a_i^2\exp(\beta a_i\xi + b_i)}{\sum_{i\in [n]}\exp(\beta a_i\xi + b_i)} - \frac{\beta(\sum_{i\in [n]} a_i\exp(\beta a_i\xi + b_i))^2}{(\sum_{i\in [n]}\exp(\beta a_i\xi + b_i))^2}\\
\frac{d^3}{d\xi^3}F_{a,b,\beta}(\xi) &= \frac{\beta^2\sum_{i\in [n]}a_i^3\exp(\beta a_i\xi + b_i)}{\sum_{i\in [n]}\exp(\beta a_i\xi + b_i)} - \frac{3\beta^2(\sum_{i\in [n]}a_i\exp(\beta a_i\xi + b_i))(\sum_{i\in [n]}a_i^2\exp(\beta a_i\xi + b_i))}{(\sum_{i\in [n]}\exp(\beta a_i\xi + b_i))^2}\\
&\qquad\qquad\qquad + \frac{2\beta^2(\sum_{i\in [n]} a_i\exp(\beta a_i\xi + b_i))^3}{(\sum_{i\in [n]}\exp(\beta a_i\xi + b_i))^3}.
\end{align*}
The crucial observation is that 
\begin{equation}\label{eq:dev-bound}
\bigg|\frac{d^3}{d\xi^3}F_{a,b,\beta}(\xi)\bigg|\le 6\beta^2 \snorm{a}_{\infty}^3.  
\end{equation}

We now prove that $F_{a,b,\beta}$ is insensitive to the precise nature of the disorder given that $\snorm{a}_{\infty}$ is sufficiently small.

\begin{lemma}\label{lem:exchange}
Let $\xi, \xi'$ have finite third moments with $\mb{E}[\xi] = \mb{E}[\xi']$ and $\mb{E}[\xi^2] = \mb{E}[{\xi'}^2]$. Then 
\[
\big|\mb{E}[F_{a,b,\beta}(\xi)] - \mb{E}[F_{a,b,\beta}(\xi')]\big|\le \beta^2\snorm{a}_{\infty}^3\cdot \big(\bbE[|\xi|^3+|\xi'|^3]\big)
\]
\end{lemma}
\begin{proof}
By Taylor's theorem (with Lagrange error term), we have 
\[\bigg|F_{a,b,\beta}(\xi) - F_{a,b,\beta}(0) - \xi \cdot F_{a,b,\beta}'(0) - \frac{\xi^2 F_{a,b,\beta}''(0)}{2}\bigg|\le \frac{|\xi|^3 \cdot \snorm{F_{a,b,\beta}'''}_{\infty}}{6}.\]
As $\mb{E}[\xi] = \mb{E}[\xi']$ and $\mb{E}[\xi^2] = \mb{E}[{\xi'}^2]$, using the triangle inequality and \cref{eq:dev-bound} we have that 
\begin{align*}
\big|\mb{E}[F_{a,b,\beta}(\xi)] - \mb{E}[F_{a,b,\beta}(\xi')]\big| &\le\beta^2 \snorm{a}_{\infty}^3 \mb{E}[|\xi|^3 + |\xi'|^3].\qedhere
\end{align*}
\end{proof}

We will apply \cref{lem:exchange} to restricted free energies on subsets of $\cS_N$ with designated ``large'' and ``small'' coordinates (corresponding to the decomposition of Section~\ref{sec:decomp}). Given a subset $I\subseteq [N]$, we let $\vJ^{(p)}[I^j]$ be the $p$-tensor with $(i_1,\dots,i_p)$-entry equal to $J_{i_1,\dots,i_p}$ if exactly $j$ of $i_1,\dots,i_p$ lie in $I$, and otherwise $0$.
Thus $\vJ=\sum_{j=0}^p \vJ^{(p)}[I^j]$.

\begin{lemma}\label{lem:interpolation-on-sphere}
Fix $T,M\ge 1$ and $w\in\cB_N$ and $I\subseteq [N]$. Let $\mu\in\cP(\cS_N)$ be a non-zero finite measure such that for $\mu$-almost all $x$, we have $x_i=w_i$ if $i\in I$ and $|x_i|\leq T$ if $i\notin I$.
Then for $\vJ$ as in \eqref{eq:truncated-disorder-properties},
\[
\big|\mb{E}[F_{\beta}(\vJ^{(p)}[I^j];\mu)]-\mb{E}[F_{\beta}(\vG^{(p)}[I^j];\mu)]\big| 
\ll_P \beta^2 M N^{3/2}(T^{3}/\sqrt{N})^{p-j}.
\]
\end{lemma}
\begin{proof}
Note that $\mu$ can be normalized to a probability measure without loss of generality. By a routine approximation argument we may assume $\mu$ has finite support. Furthermore note that $\mb{E}[|J_{i_1,\ldots,i_p}|^3]\le \oM\mb{E}[|J_{i_1,\ldots,i_p}|^2]\ll_{P} M$ for each $i_1,\dots,i_p\in [N]$. We will interpolate in two steps; we first interpolate to $\wt{G}$ which is Gaussian with appropriate symmetry constraints such that $\mb{E}[|J_{i_1,\ldots,i_p}|^2] = \mb{E}[|\wt{G}_{i_1,\ldots,i_p}|^2]$ and then interpolating these Gaussian to the appropriate variances.  

We apply \cref{lem:exchange} to each entry in $\vJ^{(p)}[I^j]$. There are at most $\ll_p |I|^j\cdot N^{(p-j)}$ such terms. The error incurred by \cref{lem:exchange} is at most
\begin{align*}
&\ll_p N^{-3(p-1)/2}\beta^2M\sum_{\substack{i_1,\ldots,i_j\in I\\ i_{j+1},\ldots,i_p\in [N]\setminus I}}\prod_{\ell=1}^j |w_{i_{\ell}}|^3\cdot T^{3(p-j)}\ll_p  N^{-3(p-1)/2}\beta^2M\cdot \snorm{w}_3^{3j}\cdot (T^3N)^{(p-j)}\\
&\ll_pN^{3/2}\beta^2M \cdot (\snorm{w}_{\infty} \cdot \snorm{w}_2^{2})^{j}\cdot (T^3N)^{(p-j)}N^{-3p/2}\ll_p \beta^{2}M N^{3/2}(T^{3}/\sqrt{N})^{(p-j)}. \qedhere
\end{align*}
\end{proof}

\begin{remark}
\label{rem:p-j=1}
    \cref{lem:interpolation-on-sphere} is useful when $p-j\geq 2$, as then the error term is sublinear in $N$.
    The case $p - j = 0$ is also easily addressed since 
    \[
    F_{\beta}(\vJ^{(p)}[I^p])
    =
    N^{-(p-1)/2}
    \la \vJ^{(p)}[I^p],w^{\otimes p}\ra.
    \]
    This is a weighted sum of the bounded, independent entries of $\vJ^{(p)}[I^p]$ and thus concentrates well; we handle these terms by a union bound over the sub-exponentially many points $w\in\cN_{T,\iota}$.
    
    The case $p-j = 1$ is more subtle, and indeed causes our results to be false in the Ising case.
    Even in the pure $p=1$ model, one has
    \[
    \sup_{\eps\in \{\pm 1\}^N}\sum_{i=1}^{N}\eps_i\xi_i = \sum_{i=1}^{N}|\xi_i|
    \approx 
    N\mb{E}[|\xi|]
    \]
    with high probability, which may not be universal.
    We address this using rotational symmetry of $\cS_N$.
\end{remark}

\subsection{Concentration}
The next standard bound shows $F_{\beta}(\vJ;\mu)$ concentrates for any $\mu\in\cP(\mc{B}_N)$.

\begin{lemma}\label{lem:concen-bound}
There exists a constant $C = C_{\ref{lem:concen-bound}}>0$ such that the following holds. Let $\mu$ be any probability measure on $\mc{B}_N$. Let $\vJ$ have independent entries with $\|\vJ\|_{\infty}\le \oM$ almost surely. 
Then
\[
\mb{P}\Big[\big|F_{\beta}(\vJ;\mu) - \mb{E}[F_{\beta}(\vJ;\mu)]\big|\ge t \Big] \le C\exp\Big(-\frac{t^2}{C\oM^2N}\Big)
.
\] 
Furthermore if $\vG$ denotes Gaussian disorder, then
\[
\mb{P}\Big[\big|F_{\beta}(\vG;\mu) - \mb{E}[F_{\beta}(\vG;\mu)]\big|\ge t \Big] \le C\exp\Big(-\frac{t^2}{CN}\Big)
.
\] 
\end{lemma}
\begin{proof}
In the first case, we apply Talagrand's inequality (see e.g. \cite[Corollary~1.2]{Coo12}). Note that the convex function $\vJ\mapsto F_{\beta}(\vJ;\mu)$ has Lipschitz constant (with respect to the Euclidean distance on $\vJ$) at most $N^{-(p-1)/2}\sup_{\sigma\in \mc{B}_N}\snorm{\sigma^{\otimes p}}_2 = \sqrt{N}$. Since $\|\vJ\|_{\infty}\le \oM$ almost surely, the result follows.

In the second case, an identical argument using concentration of Lipschitz functions of Gaussians (see e.g. \cite[(1.5)]{LT11}) completes the proof. 
\end{proof}

\section{Universality on Delocalized Subspheres}\label{sec:bound-dis}

We now prove universality of the free energy on each delocalized subsphere under uniformly bounded disorder. This is the main part of the proof.
We derive \cref{thm:main,thm:bai-yin-var} in the next subsection by combining this with a priori concentration and the preceding truncations.

\subsection{Lower bound on free--energy from delocalized vectors}
Let $\mu_{\text{Deloc}}$ be the uniform measure on the set
\begin{equation}
\label{eq:S-deloc}
\mc{S}_{\text{Deloc}} := \mc{S}_N \cap \big\{x\in \mb{R}^{N}:\snorm{x}_{\infty}\le N^{1/32}\big\}
.
\end{equation} 
Since $\mu_N(\mc{S}_{\text{Deloc}})\geq 1-N^{-\omega(1)}$, all Borel $S\subseteq\cS_N$ satisfy
\begin{equation}
\label{eq:mu-deloc-simple}
    \mu_{\text{Deloc}}(S)\leq (1+N^{-\omega(1)})\mu(S).
\end{equation}

We first lower bound $F_{\beta}(\vec{J})$ for bounded disorder by considering $\mc{S}_{\text{Deloc}}$. The idea is that by \eqref{eq:mu-deloc-simple} and rotational invariance, $\mu$ and $\mu_{\text{Deloc}}$ have essentially equal free energy under Gaussian disorder.
It will be convenient to write $\vec{G}^{(\geq 2)}$ for $\big(\vec{G}^{(2)},\dots,\vec{G}^{(P)}\big)$ and similarly for $\vJ$.

\begin{lemma}\label{lem:deloc}
We have  
\begin{equation}
\label{eq:first-item}
\mb{P}[|F_{\beta}(\vec{G}) - F_{\beta}(\vec{G}; \mu_{\text{Deloc}})|\ge N^{-1}]\le N^{-\omega(1)}.
\end{equation}
Additionally for disorder $\vJ$ as in \eqref{eq:truncated-disorder-properties}, we have 
\begin{equation}
\label{eq:second-item}
\mb{P}[|F_{\beta}(\vec{G}) - F_{\beta}(\vec{J}; \mu_{\text{Deloc}})|\ge C(\beta)N^{3/4}]\le N^{-\omega(1)}.
\end{equation}
\end{lemma}
\begin{proof}
Let $O\in O(N)$ be Haar-random. When the underlying disorder is Gaussian, the rotational invariance of $H_N$ (c.f. \cref{eq:gaussian-covariance}) implies that
\[
H_N(\bsig;O):=H_N(O\bsig)
\]
agrees with $H_N$ in distribution as a function on $\cS_N$. Let $\vec G^{(p),O}$ be the associated $p$-tensor, so that 
\begin{equation}
\label{eq:G-O-def}
 H_N(\bsig;O)=\sum_{p=1}^P N^{-(p-1)/2}\la \vec G^{(p),O},\bsig^{\otimes p}\ra.
\end{equation}
Then $Z_{\beta}(\vG^{O})=Z_{\beta}(\vG)$ and \eqref{eq:mu-deloc-simple} implies that almost surely, 
\[Z_{\beta}(\vec{G}; \mu_{\text{Deloc}})\le (1+N^{-\omega(1)})Z_{\beta}(\vec{G}).\]
Using again \eqref{eq:mu-deloc-simple} with rotational symmetry gives
\[
\mb{E}[Z_{\beta}(\vec{G}^{O}; \mu_{\text{Deloc}})|\vec{G}] \ge (1-N^{-\omega(1)})Z_{\beta}(\vec{G}).\]
Combining the preceding estimates and averaging over $\vG$, we obtain
\[
\mb{E}\lt[
\frac{|Z_{\beta}(\vec{G}^{O}; \mu_{\text{Deloc}})-Z_{\beta}(\vec{G}^{O})|}{Z_{\beta}(\vec{G})}~|~\vec{G}
\rt]
\le 
N^{-\omega(1)}
\quad
\implies
\quad
\mb{E}\lt[
\frac{|Z_{\beta}(\vec{G}^{O}; \mu_{\text{Deloc}})-Z_{\beta}(\vec{G}^{O})|}{Z_{\beta}(\vec{G}^{O})}
\rt]\le N^{-\omega(1)}.
\]
If $|Z'-Z|\leq N^{-\omega(1)} Z$ for $Z>0$, it follows deterministically that $|\log(Z')-\log(Z)|\leq N^{-\omega(1)}$.
Since $\vec{G}^{O}$ and $\vec{G}$ have the same law, \eqref{eq:first-item} then follows from the preceding display by Markov.

Fixing $\vec{J}^{(1)}$, we apply \cref{lem:interpolation-on-sphere} with $w=\vec 0\in \bbR^N$ and $I=\emptyset$ to the entries of $\vJ^{(\ge 2)}$. Recalling the first line of \eqref{eq:truncated-disorder-properties}, we find:
\begin{equation}
\label{eq:5/8-bound}
\begin{aligned}
\big|
\mb{E}[F_{\beta}(\vec{J},\mu_{\text{Deloc}}) 
- 
\mb{E}[F_{\beta}(\vec{J}^{(1)},\vec{G}^{(\geq 2)},\mu_{\text{Deloc}})]
\big|
&\ll_{P} \beta^{2}N^{3/2 + 1/32}(N^{3/64}/\sqrt{N})^{2}
\\
&\ll_P C(\beta) N^{5/8}.
\end{aligned}
\end{equation}
Combining with Lemma~\ref{lem:concen-bound} almost proves \eqref{eq:second-item}: by adjusting $C(\beta)$ it implies
\begin{equation}
\label{eq:deloc-first-step}
\mb{P}\big[|F_{\beta}(\vec{J},\mu_{\text{Deloc}}) 
- 
F_{\beta}(\vec{J}^{(1)},\vec{G}^{(\geq 2)},\mu_{\text{Deloc}})|\geq C(\beta)N^{5/8}\big]
\leq 
N^{-\omega(1)}
\end{equation}
However as explained in Remark~\ref{rem:p-j=1}, the linear term requires a separate argument. 
The assumption $\snorm{\vec{J}^{(1)}}_{\infty}\le N^{1/64}$ easily implies
\begin{equation}
\label{eq:by-AH}
\mb{P}[\snorm{\vec{J}^{(1)}}_2^2\notin [N - N^{3/4}, N + N^{3/4}]] \le N^{-\omega(1)}
.
\end{equation}
Define the slight rescaling $\wt{J^{(1)}}=\vec{J}^{(1)}\sqrt{N}/\|\vec{J}^{(1)}\|_2\in\cS_N$.
Then \eqref{eq:by-AH} implies
\begin{equation}
\label{eq:rescale-J-safe}
\mb{E}\big[|F_{\beta}(\vec{J}^{(1)},\vec{G}^{(\geq 2)},\mu_{\text{Deloc}})]-\mb{E}[F_{\beta}(\wt{J^{(1)}},\vec{G}^{(\geq 2)},\mu_{\text{Deloc}})| \geq N^{3/4}\big]
\ll
N^{-\omega(1)}.
\end{equation}
Also by \eqref{eq:by-AH}, $\bbP[\snorm{\wt{J^{(1)}}}_{\infty}\le 2N^{1/64}]\geq 1-N^{-\omega(1)}$. 
On this event, we have
\begin{equation}\label{eq:compar}
\mb{P}^{v\sim \mu_N}[v\in \mc{S}_{\on{Deloc}}~|~\sang{v,\wt{J^{(1)}}} = t] \ge 1-N^{-\omega(1)},\quad\forall\, t\in [-N,N].
\end{equation}
Here we used that $10N^{1/64}\leq  N^{1/32}$ (recall \eqref{eq:S-deloc}). Namely the conditional law of $v$ given $\sang{v,\wt{J^{(1)}}} = t$ is uniform on a subsphere of $\cS_N$ with center $t\wt{J^{(1)}}/N$ obeying  $\|t\wt{J^{(1)}}/N\|_{\infty}\leq 2N^{1/64}$.

Let $\mu_{\text{Deloc},\wt{J^{(1)}},t}$ denote the (regular) conditional law of $v\sim\mu_{\text{Deloc}}$ given that $\sang{v,\wt{J^{(1)}}} = t$.
Let $\mu_{\wt{J^{(1)}},t}$ denote the conditional law of $v\sim\mu_{N}$ given that $\sang{v,\wt{J^{(1)}}} = t$.
We proceed similarly to the proof of \eqref{eq:first-item}, and let $O$ be a uniformly random orthogonal matrix conditioned to fix $\wt{J^{(1)}}$ (i.e. $O$ acts by a Haar-random orthogonal matrix for the subspace $(\wt{J^{(1)}})^{\perp}\subseteq\bbR^N$).
Then \cref{eq:compar} and rotational invariance within $(\wt{J^{(1)}})^{\perp}$ imply that for all $\snorm{\wt{J^{(1)}}}_{\infty}\le 2N^{1/64}$ and $t\in [-N,N]$:
\[
\mb{E}^{O}
\big[
\big|
Z_{\beta}(\wt{J^{(1)}},\vec{G}^{(\geq 2),O},\mu_{\text{Deloc},\wt{J^{(1)}},t})
-
Z_{\beta}(\wt{J^{(1)}},\vec{G}^{(\geq 2),O},\mu_{\wt{J^{(1)}},t})
\big|
\big]
\le 
N^{-\omega(1)}
Z_{\beta}(\wt{J^{(1)}},
\vec{G}^{(\geq 2)},\mu_{\wt{J^{(1)}},t})
.
\]
Here the notation $\mb{E}^{O}$ indicates that we have conditioned on $(\wt{J^{(1)}},\vec{G}^{(\ge 2)})$.
Integrating over $t$ and dividing by 
$Z_{\beta}(\wt{J^{(1)}},\vec{G}^{(\geq 2)})$ shows that if $\snorm{\wt{J^{(1)}}}_{\infty}\le 2N^{1/64}$, then
\[
\mb{E}^{O}
\lt[
\frac{
\big|
Z_{\beta}(\wt{J^{(1)}},\vec{G}^{(\geq 2),O},\mu_{\text{Deloc}})
-
Z_{\beta}(\wt{J^{(1)}},\vec{G}^{(\geq 2),O})
\big|
}{Z_{\beta}(\wt{J^{(1)}},
\vec{G}^{(\geq 2)})}
\rt]
\le 
N^{-\omega(1)}
.
\]
As in the end of the proof of \eqref{eq:first-item}, Markov now implies that if $\snorm{\wt{J^{(1)}}}_{\infty}\le 2N^{1/64}$, then
\begin{equation}
\label{eq:deloc-almost-done}
\bbP^{\vG}[
|F_{\beta}(\wt{J^{(1)}},\vec{G}^{(\geq 2)},\mu_{\text{Deloc}})
-
F_{\beta}(\wt{J^{(1)}},\vec{G}^{(\geq 2)})|\geq N^{-1}
]
\le N^{-\omega(1)}.
\end{equation}
Furthermore $\bbP\big[\big|\snorm{G^{(1)}}_2 - \sqrt{N}\big|\ge \lambda N^{1/5}\big]\le N^{-1}e^{-\lambda^2}$ for $\lambda\geq 1$.
By rotational invariance of $G^{(\geq 2)}$ we may without loss of generality couple $(\vJ,\vG)$ so that $\vG^{(1)}/\|\vG^{(1)}\|=\vJ^{(1)}/\|\vJ^{(1)}\|$ when comparing $F_{\beta}(\wt{J^{(1)}},\vec{G}^{(\geq 2)})$ and $F_{\beta}(\vec{G})$; from this observation we find:
\[
\bigg|\mb{P}[|F_{\beta}(\wt{J^{(1)}},\vec{G}^{(\geq 2)}) - F_{\beta}(\vec{G})|\geq N^{3/4}]\bigg|
\le
\bbP\big[\big|\snorm{G^{(1)}}_2 - \sqrt{N}\big|\ge  N^{1/4}\log N\big]
\ll 
N^{-\omega(1)}.
\]
Combining with \eqref{eq:deloc-first-step}, \eqref{eq:rescale-J-safe} and \eqref{eq:deloc-almost-done} finally completes the proof of \eqref{eq:second-item}.
\end{proof}

As $F_{\beta}(\vec{J}; \mu)\geq F_{\beta}(\vec{J}; \mu_{\text{Deloc}})-o(1)$ almost surely, \cref{lem:deloc} essentially implies the lower bounds in \cref{thm:main,thm:bai-yin-var}. For the upper bound, we use a similar strategy with $\mu_{\text{Deloc}}$ replaced by $w + \mu_w$ from \cref{lem:rnd-point}.

\subsection{Upper bound on free-energy}\label{subsec:upper}

We have the following estimate on certain projections of $\nabla H_N(w)$.
For a subspace $U\subseteq \bbR^N$, we write $P_U^{\perp}$ for the orthogonal projection onto $U^{\perp}$.

\begin{lemma}\label{lem:grad-estim}
Fix $T\in [\log N, N^{1/8}]$ and $\iota = N^{-16P}$. Let $(w,\mu_w)$ be as in \cref{lem:rnd-point} and set $I_w = \on{supp}(w)\subseteq [N]$. Suppose $\vec{J}$ is either as in \eqref{eq:truncated-disorder-properties} or is Gaussian. Then for each $w\in \mc{N}_{T,\iota}$, 
\[\mb{P}\bigg[\bigg|\snorm{P_{I_w}^{\perp}\nabla H_N(w)}_2 - N^{1/2}\sqrt{\xi'(\snorm{w}_2^2/N)}\bigg|\ge N^{1/2}(\log N)/T^2 + N^{3/10}\bigg]\le N^{-\omega(1)}.\]
\end{lemma}
\begin{proof}
We explicitly write
\begin{equation}
\label{eq:projected-grad-explicit}
\snorm{P_{I_w}^{\perp}\nabla H_N(w)}_2^2=\sum_{p=1}^P\gamma_p^2\Big(\frac{1}{N}\Big)^{p-1}\sum_{i_p\in I_w^c}p\Big(\sum_{i_1,\dots,i_{p-1}\in [N]}J_{i_1,\dots,i_p} w_{i_1}\dots w_{i_{p-1}}\Big)^2.
\end{equation}
Here we used the $\mf{S}_p$-symmetry of $\vJ^{(p)}$ to group terms. As $|I_w^c|\ge N(1-1/T^2)$, 
\[
1-O_P(T^{-2})
\le 
\frac{\mb{E}[\snorm{P_{I_w}^{\perp}\nabla H_N(w)}_2^2]}{N \xi'(\snorm{w}_2^2/N)}
\le 
1.
\]
Next we show the right-hand side of \eqref{eq:projected-grad-explicit} concentrates. 
For each value of $i_p$, the inner sum has mean zero and variance $\|w\|_2^{2(p-1)}\le N^{p-1}$. Thus the inner sums are subgaussian with standard deviation proxy $O(N^{1/4}\cdot N^{(p-1)/2})$. Then by Bernstein's inequality, the outer sum is within $N^{4/5}$ of its expectation except with probability $\exp(-N^{-\Omega(1)})$. 
\end{proof}

We now upper bound the expected free energy on each $\mu_w$.
\begin{lemma}\label{lem:invar-upper-bound}
Fix $T\in [\log N, N^{1/8}]$ and $\iota = N^{-16P}$. Let $(w,\mu_w,\mu_w^{\ast})$ be as in \cref{lem:rnd-point}. If $\vJ$ is as in \eqref{eq:truncated-disorder-properties}, then 
\[\mb{E}[F_{\beta}(\vec{J};w + \mu_w)]- \mb{E}[F_{\beta}(\vec{G};w+\mu_w^{\ast})] \ll_P C(\beta) (N^{9/16} T^{6}) + N(\log N)/T^2 + N^{4/5}.\]
\end{lemma}
\begin{proof}
As before, set $I_w = \on{supp}(w)$.
We have the expansion 
\[
H_N(v)
=
H_N(w)
+
\la\nabla H_N(w),v-w\ra
+
\sum_{\substack{1\leq p\leq P,\\ 0\leq j\leq p-2}}
\binom{p}{j}\Big\la \vJ^{(p)}[I_w^j], w^{\otimes j}\otimes (v-w)^{\otimes (p-j)}\Big\ra.
\]
Since $\mb{E}[H_N(w)] = 0$, \cref{lem:concen-bound} (with $\mu$ a point mass at $w$) implies $\bbP^{\vJ}[|H_N(w)|\le N^{2/3}]\ge 1-N^{-\omega(1)}$. Writing $v' = v-w\sim \mu_w$, it suffices to understand the distribution of 
\[
\la\nabla H_N(w), v'\ra
+
\sum_{\substack{1\leq p\leq P,\\ 0\leq j\leq p-2}}
\binom{p}{j}\Big\la \vJ^{(p)}[I_w^j], w^{\otimes j}\otimes (v')^{\otimes (p-j)}\Big\ra.
\]
Note that $\la\nabla H_N(w), v'\ra$ depends only on $P_{I_w}^{\perp}\nabla H_N(w)$ since $\on{supp}(v')\subseteq [N]\setminus I_w$.
In turn, $P_{I_w}^{\perp}\nabla H_N(w)$ depends only on $\big(\vJ^{(p)}[I_w^{p-1}]\big)_{1\leq p\leq P}$, \ie is independent of the second term above. 
In light of \cref{lem:grad-estim}, we may replace $\nabla H_N(w)$ by a deterministic $z\in \mb{R}^{[N]\setminus I_w}$ with $\snorm{z}_2 = N^{1/2}\sqrt{\xi'(\snorm{w}_2^2/N)}$.
We then obtain the modified and recentered Hamiltonian 
\[
\wt{H_N}(v') := \la z, v'\ra
+
\sum_{\substack{1\leq p\leq P,\\ 0\leq j\leq p-2}}
\binom{p}{j}\Big\la \vJ^{(p)}[I_w^j], w^{\otimes j}\otimes (v')^{\otimes (p-j)}\Big\ra.
\]
Define the associated free energy, for general $\mu\in\cP(\cS_w)$, to be
\[\wt{F}_\beta(\vec{J},z;\mu) := \frac{1}{\beta}\log\int_{\bsig\in \mc{B}_N}\exp(\beta\wt{H_N}(\bsig))d\mu(\bsig).\]
Combining \cref{lem:concen-bound,lem:grad-estim} we bound the free energy error incurred by this replacement: 
\begin{equation}
\label{eq:error-replace-z}
\bbE[|F_{\beta}(\vJ,w+\mu_w)-\wt{F}_{\beta}(J,z,\mu_w)|]
\ll 
N(\log N)/T^2 + N^{4/5}.
\end{equation}
Iteratively applying \cref{lem:interpolation-on-sphere} with $I=I_w$ shows that for any fixed $z$ as above,
\begin{equation}
\label{eq:interpolate-given-z}
\big|\mb{E}[\wt{F}_\beta(\vec{J},z;\mu_w)] - \mb{E}[\wt{F}_\beta(\vec{G},z;\mu_w)]\big|
\ll_P 
C(\beta) (N^{9/16} T^{6}).
\end{equation}
Moreover \cref{lem:rnd-point}\ref{it:rnd-point-4} implies that almost surely,
\begin{equation}
\label{eq:a-s-bound}
\wt{F}_\beta(\vec{G},z;\mu_w) 
\le 
\wt{F}_\beta(\vec{G},z;\mu_w^{\ast})
+
N^{-\omega(1)}.
\end{equation}

Recall that $\mu_w^{\ast}$ is uniform on the sphere $\cS_w$ of radius $\sqrt{N - \snorm{w}_2^2}$ on coordinates $[N]\setminus I_w$. Therefore if $z,z'\in \mb{R}^{[N]\setminus I_w}$ satisfy $\snorm{z}_2 = \snorm{z'}_2$ then by orthogonal invariance,
\[
\mb{E}[\wt{F}_\beta(\vec{G},z;\mu_w^{\ast})] 
= 
\mb{E}[\wt{F}_\beta(\vec{G},z';\mu_w^{\ast})].
\]
This implies that for all $z\in \mb{R}^{[N]\setminus I_w}$ with $\snorm{z}_2 = N^{1/2}\sqrt{\xi'(\snorm{w}_2^2/N)}$: 
\begin{equation}
\label{eq:round-gaussian}
\bigg|\mb{E}[\wt{F}_\beta(\vec{G},z;\mu_w^{\ast})] - \mb{E}[\wt{F}_{\beta}(\vec{G},\vec{G}^{(1)},\mu_w^{\ast})]\bigg|
\ll_P 
N(\log N)/T^2 + N^{4/5}.
\end{equation}
Here we used that $|H_N(w)|\leq N^{2/3}$ with probability $1-N^{-\omega(1)}$ and the norm concentration of $P_{I_w}^{\perp}\nabla H_N(w)$ provided by \cref{lem:grad-estim}, both for Gaussian disorder $\vG$. 
Since we have by definition $\wt{F}_{\beta}(\vec{G},\vec{G}^{(1)},\mu_w^{\ast})=F_{\beta}(\vec G,w+\mu_w^{\ast})$, combining \eqref{eq:error-replace-z}, \eqref{eq:interpolate-given-z}, \eqref{eq:a-s-bound}, \eqref{eq:round-gaussian} completes the proof.
\end{proof}

\section{Finishing the proof}\label{sec:finish}
We now combine the previous sections to obtain our main results.

\subsection{Universality at Positive Temperature}

First we deduce the main convergence statements of \cref{thm:main,thm:bai-yin-var} for $\beta<\infty$.
We defer the full proof of quantitative rates in the former to Subsection~\ref{subsec:quantitative-rates-proof}, and necessity of moment bounds in the latter to Subsection~\ref{subsec:moment-bounds-necessary}.

\begin{corollary}
\label{cor:weak-universality-finite-beta}
    Fix $\beta\in [0,\infty)$. In the setting of \cref{thm:main}, convergence in probability holds:
    \[
    \plim_{N\to\infty} F_{\beta}(\vJ)/N=\Par(\xi,\beta).
    \]
    In the setting of \cref{thm:bai-yin-var}, one has the almost sure limit:
    \[
    \lim_{N\to\infty} F_{\beta}(\vJ)/N=\Par(\xi,\beta).
    \]
\end{corollary}

\begin{proof}
    Let $c\geq \Omega_{P}(\eps)>0$ be a small constant which may vary from line to line.
    In the setting of \cref{thm:main}, applying \cref{lem:tail-rem,lem:large-rem,lem:bound,lem:diff-bound} shows 
    \[
    \bbP[|F_{\beta}(\vJ)-F_{\beta}(\vJ^{\msf{mod}})|\leq N^{1-c}]\geq 1-N^{1-c}.
    \]
   Here we have used that $\wt\delta\leq N^{-c}$ in the setting of \cref{thm:main}. Similarly in the setting of \cref{thm:bai-yin-var} the error is $o(N)$. Thus we focus on $F_{\beta}(\vJ^{\msf{mod}})$, which obeys \eqref{eq:truncated-disorder-properties}.
    The lower bound follows by \eqref{eq:second-item}.

    For the upper bound, taking $T= N^{0.01}$, \cref{lem:concen-bound,lem:invar-upper-bound} imply that for fixed $w\in\cN_{T,\iota}$, 
    \[
    \bbP[F_{\beta}(\vJ^{\msf{mod}};w+\mu_w)\leq F_{\beta}(\vG;w+\mu_w^{\ast})+N^{1-c}]\geq 1-e^{-N^{1-2c}}.
    \]
    Since $\log|\cN_{T,\iota}|\leq N^{1-c}$, it follows from \cref{lem:meas-decomp} that
    \[
    F_{\beta}(\vJ^{\msf{mod}})
    \leq 
    N^{1-c}
    +
    \max_{w\in\cN_{T,\iota}}
    \Big(F_{\beta}(\vJ^{\msf{mod}};w+\mu_w)
    +
    \log\mu_N(\cC_w)
    \Big)
    .
    \]
    Combining, and using again \cref{lem:concen-bound,lem:invar-upper-bound}, it remains to show that for each $w\in\cN_{T,\iota}$,
    \[
    \bbP[F_{\beta}(\vG;w+\mu_w^{\ast})
    +
    \log\mu_N(\cC_w)
    \leq 
    F_{\beta}(\vG)+N^{1-c}]\stackrel{?}{\geq} 1/2.
    \]
    Noting that $\log \mu_{N}(\cC_w)\le 0$, it suffices to prove that 
    \[
    \mb{E}[F_{\beta}(\vG;w+\mu_w^{\ast})]\le \mb{E}[F_{\beta}(\vG)].
    \]
    For $O$ be a random matrix drawn from Haar measure on the orthogonal group, and let $\vec{G}^{O}$ be defined via 
    \begin{equation}\label{eq:G-O-def}
    H_N(\bsig;O):=H_N(O\bsig)= \sum_{p=1}^P N^{-(p-1)/2}\la \vec G^{(p),O},\bsig^{\otimes p}\ra.
    \end{equation}
    Via orthogonal invariance of the law of $\vec{G}$, we have that 
\[
\mb{E}[Z_{\beta}(\vec G^{(O)};w+\mu_w^{\ast})]= \mb{E}^{O}[Z_{\beta}(\vec G^{(O)};w+\mu_w^{\ast})~|~\vec G]=\mb{E}[Z_{\beta}(\vec G)].
\]
The result then follows via Jensen's inequality conditionally on $\vec G$ applied to $F_{\beta}(\vec G^{(O)};w+\mu_w^{\ast})=\log(Z_{\beta}(\vec G^{(O)};w+\mu_w^{\ast}))$. Combining the obtained estimates completes the proof, using Borel--Cantelli for the almost-sure convergence in Theorem~\ref{thm:bai-yin-var}.
\end{proof}

\begin{remark}
\label{rem:lee-yin}
\cref{prop:lee-yin} follows exactly as above, modulo slight modifications in \cref{sec:trunc}.
Namely in \eqref{eq:the-sum-which-is-finite}, one instead finds by diagonalization a decreasing sequence $\delta_N^{\LY}\downarrow 0$ such that 
\[
\lim_{N\to\infty}
\bbP[\sup_{1\leq p\leq P}\sup_{i_1,\dots,i_p\in [N]}|J_{i_1,\dots,i_p}|\geq \delta_N^{\LY}\sqrt{N}]=0.
\]
Note that \eqref{eq:external-field-BY-truncation} does not require any change since the external field entries still have finite variance (which is trivially necessary to have $F_{\beta}(\vJ)\leq O(N)$).
Then as in \cref{lem:tail-rem}, we find that $\lim_{N\to\infty} \oGS(\vJ^{\msf{tail}})/N=0$.
Next, \cref{lem:large-rem} uses moment bounds only via Markov's inequality, hence goes through unchanged.
The same is true for \cref{lem:bound}.
Similarly \eqref{eq:Markov-2} still applies with arbitrarily small adjustment in the exponent, which suffices for all its uses.

Necessity of \eqref{eq:lee-yin} follows routinely from events of large individual entries, as in \cite[Section 4]{lee2014necessary}.
\end{remark}

\subsection{Universality of the Ground State}

Interpolation requires $\beta<\infty$; the following lemma reduces the case $\beta=\infty$ to large finite $\beta$. This is easier in the Gaussian case thanks to apriori up-to-constants bounds on the ground state, see \cite[Lemma 8]{CS17}.
Note that this issue is irrelevant in the Ising case since there $|F_{\beta}(\vJ)-GS(\vJ)|\leq N\log(2)/\beta$ holds deterministically.

\begin{proposition}[{\cite{Ban38}}]
\label{prop:banach}
For any symmetric tensor $\vJ^{(p)}\in \bbR^{N^p}$, we have 
    \[
    \oGS(\vJ)
    =
    N^{-(p-1)/2}
    \max_{\bsig_1,\dots,\bsig_p\in\cS_N}
    \la 
    \vJ^{(p)}
    ,
    \bsig_1\otimes\dots\otimes \bsig_p\ra.
    \]
\end{proposition}

\begin{lemma}
\label{lem:positive-to-zero-temp}
There exists an absolute constant $C$ such that for any $p$ and large enough $N$, for any $\vJ\in \bbR$ and $\beta\in [1,\infty)$, 
\[\GS(\vJ)-\frac{CN}{\beta}\log\lt(2+\frac{P\beta\oGS(\vJ)}{N}\rt)
\leq F_{\beta}(\vJ)\leq \GS(\vJ).\]
\end{lemma}

\begin{proof}
    The latter bound is obvious.
    For the former, we begin with the deterministic bound
    \begin{equation}
    \label{eq:HN-Lipschitz}
    \|\nabla H_N^{(p)}(\bx)\|
    =
    pN^{-(p-1)/2}
    \|\la \vJ^{(p)},\bx^{\otimes (p-1)}\ra\|
    {\leq}
    pN^{-1/2}\oGS(\vJ^{(p)})
    \end{equation}
    for all $\bx\in \cB_N$.
    The latter follows by applying Proposition~\ref{prop:banach} to the symmetrization of $\vJ$.
    Summing, 
    \begin{equation}
    \label{eq:gradient-bound}
    \|\nabla H_N(\bx)\|
    \leq 
    PN^{-1/2}\oGS(\vJ).
    \end{equation}
    Let the maximizer $\bsig_*\in\cS_N$ satisfy $H_N(\bsig_*)=\GS(\vJ)$.
    With $r=\lt(1+\frac{P\beta^2\oGS(\vJ)}{N}
    \rt)^{-1}$
    let 
    \[
    S_r
    =
    \{\bsig\in\cS_N~:~\|\bsig-\bsig_*\|\leq r\sqrt{N}\}.
    \]
    Then for all $\bsig\in S_r$, \eqref{eq:HN-Lipschitz} gives:
    \[
    \beta H_N(\bsig)
    \geq 
    \beta\GS(\vJ)
    -
    \frac{N}{\beta}.
    \]
    Therefore
    \begin{align*}
    F_{\beta}(\vJ)
    &\geq 
    \frac{1}{\beta}
    \log\int_{S_r}
    \exp(\beta H_N(\bsig))
    ~\de\bsig
    \\
    &\geq 
    \GS(\vJ)
    -
    \frac{N}{\beta}
    +
    \frac{\log \mu_{\cS_N}(S_r)}{\beta}
    \\
    &\geq 
    \GS(\vJ)
    -
    \lt(\frac{N}{\beta}
    +
    \frac{CN\log r}{\beta}\rt).
    \qedhere
    \end{align*}
\end{proof}

\begin{corollary}
\label{cor:positive-to-zero-temp}
    \cref{cor:weak-universality-finite-beta} holds for $\beta=\infty$.
\end{corollary}

\begin{proof}
    It is standard that $\Par(\xi,\beta):[0,\infty]\to [0,\infty)$ is continuous, increasing and bounded as a function of $\beta$ (see e.g. \cite{CS17}). 
    Applying \cref{cor:weak-universality-finite-beta} at large finite $\beta$ in the pure case to $\pm \vJ^{(p)}$, \cref{lem:positive-to-zero-temp} implies that $\oGS(\vJ^{(p)})\leq CN$ with probability $1-o(1)$.
    (Here we use the fact that either $\GS(\vJ)=\oGS(\vJ)$ or $\GS(-\vJ)=\oGS(\vJ)$.)
    Hence the same holds for $\oGS(\vJ)$ with general $\xi$ by the triangle inequality.
    Applying Lemma~\ref{lem:positive-to-zero-temp} for general $\xi$ and large $\beta<\infty$ finishes the proof.
\end{proof}

\subsection{Polynomially High Probability and Expectation Bounds}
\label{subsec:quantitative-rates-proof}

Here we obtain the quantitative bounds claimed in \cref{thm:main}.

\begin{proof}[Proof of \cref{thm:main}]
    Let $\wh \delta=\wh\delta_N=N^{-\frac{\eps}{100P}}$ and, similarly to before, decompose each entry as 
    \begin{align*}
    \vJ_{i_1,\dots,i_p}
    &=
    \vJ^{\msf{small}}_{i_1,\dots,i_p}
    +\underbrace{
    \big(
    \vJ_{i_1,\dots,i_p}\,\mbm{1}_{M/2\leq |\vJ_{i_1,\dots,i_p}|\leq \wh \delta\sqrt{N}}
    -\bbE[
    \vJ_{i_1,\dots,i_p}\, \mbm{1}_{M/2\leq|\vJ_{i_1,\dots,i_p}|\leq \wh \delta\sqrt{N}}
    ]
    \big)
    }_{\vJ_{i_1,\dots,i_p}^{(A)}}
    \\
    &\quad\quad
    -
    \underbrace{
    \big(
    \bbE[
    \vJ_{i_1,\dots,i_p}
    \,
    \mbm{1}_{|\vJ_{i_1,\dots,i_p}|> \wh \delta\sqrt{N}}
    ]
    \big)
    }_{\vJ_{i_1,\dots,i_p}^{(B)}}
    +
    \underbrace{
    \big(\vJ_{i_1,\dots,i_p}\cdot \mbm{1}_{|\vJ_{i_1,\dots,i_p}|> \wh \delta\sqrt{N}}\big)
    }_{\vJ_{i_1,\dots,i_p}^{(C)}}.
    \end{align*}

    The proof of \cref{cor:weak-universality-finite-beta} shows that in fact,
    \[
    \bbP[|F_{\beta}(\vJ)-\bbE[F_{\beta}(\vG)]|\leq N^{1-c}]\geq 1-N^{-c}
    \]
    for some $c\geq \Omega_P(\eps)$.
    Note that if we simply redefine $P$ to $25P$ (by setting $\gamma_{P+1}=\dots=\gamma_{25P}=0$) then $\vJ^{(A)}$ becomes $\vJ^{\msf{large}}+\sum_{j} \vJ^{\msf{scale}(j)}$.
    Hence \cref{lem:large-rem,lem:bound} imply
    \[
    \bbP[|\oGS(\vJ^{(A)})|\leq N^{1-c}]\geq 1-N^{-c}.
    \]
    Note $\|\vJ^{\msf{small}}+\vJ^{(A)}\|_{\infty}\leq N^{\frac{1}{2}-c}$ almost surely. Thus \cref{lem:concen-bound} implies (by slightly adjusting $c$):
    \begin{equation}
    \label{eq:talagrand-J-A}
    \bbP[|F_{\beta}(\vJ^{\msf{small}}+\vJ^{(A)})-\bbE[F_{\beta}(\vG)]|\geq (t+1)N^{1-c}]
    \leq 
    e^{-t^2 N^{1-3c}},\quad\forall t\geq 0
    .
    \end{equation}
    We next show $\vJ^{(B)}$ and $\vJ^{(C)}$ are negligible. 
    By \eqref{eq:Markov-2} with $2p$ replaced by $2p+\eps$, we have 
    \[
    \|\vJ^{(p),(B)}\|_{\infty}
    \leq 
    CN^{-p+\frac{1-\eps}{2}}N^{\frac{\eps}{50}}
    \leq 
    CN^{-p + \frac{1}{2} - \frac{\eps}{3}}.
    \]
    Then almost surely by Cauchy--Schwarz, for all $\bsig\in\cS_N$:
    \[
    N^{-(p-1)/2}\la \vJ^{(p),(B)},\bsig^{\otimes p}\ra
    \leq 
    N^{-(p-1)/2} \la \vJ^{(p),(B)},\vJ^{(p),(B)}\ra^{1/2} \|\bsig\|^{p}
    \leq 
    CN^{1-\frac{p}{2}-\frac{\eps}{3}}.
    \]
    Thus $\vJ^{(B)}$ has deterministically negligible contribution.
    For $\vJ^{(C)}$, we similarly have
    \begin{align*}
    \big|\bbE\big[
    (\vJ_{i_1,\dots,i_p}^{(p),(C)})^2\big]\big|
    &\ll
    N^{-p + 1 - \frac{\eps}{3}}
    \\
    \implies 
    \bbE[\|\vJ^{(p),(C)}\|_2^2]
    &\ll 
    N^{1 - \frac{\eps}{3}}
    \\
    \implies 
    N^{-(p-1)/2}
    \,
    \bbE[\sup_{\bsig\in\cS_N} \la \vJ^{(p),(C)},\bsig^{\otimes p}\ra]
    &\ll
    N^{1-\frac{\eps}{6}}.
    \end{align*}
    Combining the bounds on $\vJ^{(B)}$ and $\vJ^{(C)}$, we find that 
    \[
    \bbE[F_{\beta}(\vJ)-F_{\beta}(\vJ^{\msf{small}}+\vJ^{(A)})]\ll_P N^{1-\frac{\eps}{6}},\quad\forall\, \beta\in [0,\infty].
    \]
    Combining with \eqref{eq:talagrand-J-A} completes the proof.
\end{proof}

\subsection{Necessity of Moment Bounds in Theorem~\ref{thm:bai-yin-var}}
\label{subsec:moment-bounds-necessary}

The proof here is similar to \cite{BY88}.

\begin{proof}[Proof of Theorem~\ref{thm:bai-yin-var}, Final Assertion]
We assume $\nu_p$ has infinite $2p$-th moment and show that almost surely, 
\begin{equation}
\label{eq:bai-yin-reverse}
\limsup\limits_{N\to\infty}\GS(\vJ)/N=\infty.
\end{equation}
The other direction then follows by applying this result to $-\vJ$, which satisfies the same assumptions.

Fix a large constant $\oC=\oC(p)>0$, and let $1\leq j\leq p$ be the smallest value such that $\nu_p$ has infinite $2j$-th moment.
Note that there exist $\Theta(N^j)$ sequences $1\leq i_1\leq\dots\leq i_p\leq N$ with $j$ distinct indices, and the same is true if we require $i_1<i_2$.
Then it is well-known that almost surely, there are infinitely many $J_{i_1,\dots,i_p}$ with $j$ non-zero entries, $i_1<i_2$, and 
\[
|J_{i_1,\dots,i_p}|\geq i_p^{1/2p}\cdot 10^{p^j}\cdot \oC.
\]
By minimality of $j$, this is not the case with $j$ replaced by $j-1$, even with $\oC$ replaced by $1/\oC$.
Similarly it is not the case if one requires the existence of two distinct such entries $J_{i_1,\dots,i_p}\neq J_{i_1',\dots,i_p'}$ supported in the same subset of $j$ coordinates, even with $\oC$ replaced by $1/\oC$.

Thus, consider such an entry $J_{i_1,\dots,i_p}$ with $j$ distinct values, $i_1<i_2$, and with $i_p$ sufficiently large.
Then take $\bsig\in\cS_N$ to have $\sigma_{i_1}=sign\big(J_{i_1,\dots,i_p}\big)\sqrt{N/j_*}$, and $\sigma_{i_k}=\sqrt{N/j_*}$ for $2\leq k\leq p$, with all other coordinates zero. It directly follows that 
\[
    H_{i_p}^{(p)}(\bsig)/N\geq \Omega_p(\oC).
\]
This comes from the $J_{i_1,\dots,i_p}$ contribution, whereas all other terms contribute at lower order thanks to the observations above based on minimality of $j$.

Next for each $p'\neq p$, the conditional law of $H_{i_p}^{(p')}(\bsig)/N$ does not depend on $(i_p,\bsig)$ and has uniformly bounded second moment by \eqref{eq:gaussian-covariance}. In particular it is tight. 
Therefore 
\[
\limsup\limits_{N\to\infty}\GS(\vJ)/N\geq \Omega_p(\oC)
\]
holds almost surely. 
Since $\oC$ was arbitrary, we conclude \eqref{eq:bai-yin-reverse}.
\end{proof}

\section{Extensions}
\label{sec:extensions}

Here we give extensions to multi-species models, tensor PCA, and $\ell^q$ balls for $q>2$. All follow by essentially the same proof. For simplicity we often restrict to the case of $(C,\eps)$-moment bounds.

Fix a species set $\sS = \{1,\ldots,r\}$ with weights $\vlam = (\lambda_1,\ldots,\lambda_r) \in \bbR_{>0}^\sS$. 
For $N\in\bbZ_+$, fix a deterministic partition $\{1,\ldots,N\} = \bigsqcup_{s\in\sS}\, \cI_s$ with $\lim_{N\to\infty} N_s / N =\lambda_s$ for $N_s=|\cI_s|$.
For $s\in \sS$ and $\bx \in \bbR^N$, let $\bx_s \in \bbR^{\cI_s}$ denote the restriction of $\bx$ to coordinates $\cI_s$.
The domain of a multi-species spherical spin glass is the product of spheres
\begin{equation}
\label{eq:def-cS}
    \cS_{N,r} = \lt\{
        \bx \in \bbR^N : 
        \snorm{\bx_s}_2^2 = \lambda_s N
        \quad\forall~s\in \sS
    \rt\}
    \simeq \prod_{s=1}^r \cS_{N_s}.
\end{equation}
We let $\mu_{N,r}=\prod_{s=1}^r \mu_{N_s}$ be uniform measure on $\cS_{N,r}$.
For each $1\le p\le P$ fix a symmetric tensor $\Gamma^{(p)} = (\gamma_{s_1,\ldots,s_p})_{s_1,\ldots,s_p\in \sS} \in (\bbR_{\ge 0}^r)^{\otimes p}$, and let the disorder $\vJ^{(p)}\in\bbR^{N^p}$ be as in the main body.

Then a multi-species spin glass Hamiltonian takes the form
\begin{align*}
    H_N(\bsig)
    = 
	\sum_{p \ge 1}
	\fr{1}{N^{(p-1)/2}}
	\sum_{i_1,\ldots,i_k=1}^N 
	\gamma_{s(i_1),\ldots,s(i_p)} \vJ^{(p)}_{i_1,\ldots,i_p} \sigma_{i_1}\cdots \sigma_{i_p}
\end{align*}
for $\bsig = (\sigma_1,\ldots,\sigma_N) \in \cS_{N,r}$.

Just as ordinary spherical spin glasses encompass the maximum eigenvalue of a symmetric matrix, multi-species models encompass the maximum singular value of a rectangular matrix. Notably, the Parisi formula is not known for general multi-species models; even the existence of the limiting free energy is open.
See e.g. \cite{barra2015multi,panchenko2015free,subag2021tap1,subag2021tap2} for recent progress.
Nonetheless we still obtain universality.

\begin{theorem}\label{thm:multi-species}
Fix $\xi,\vlam,\Gamma$ and $\eps\in (0,1)$ and $C\ge 1$. Consider disorder $J_{i_1,\ldots,i_p}$ obeying $(C,\eps)$-moment bounds. There exists $c\geq \Omega_P(\eps)$ such that for large $N$, uniformly in $\beta\in [0,\infty]$:
\begin{align*}
    \bbE[|F_{\beta}(\vJ;\mu_{N,r})-\bbE F_{\beta}(\vG;\mu_{N,r})|]&\leq N^{1-c}.
\end{align*}
\end{theorem}

The proof is essentially identical to the single-species case.
The truncation in Section~\ref{sec:trunc} requires no modification since the species only affect the disorder up to constant factor scaling. 
Similarly the interpolation between $\vJ$ and $\vG$ works in the same way on delocalized coordinates.
One defines $\cN_{T,\iota,1},\dots,\cN_{T,\iota,r}$ as before separately within each $\cS_{N_j}$ and works on product points $(w_1,\dots,w_r)\in \prod_{j=1}^r \cN_{T,\iota,r}$.
Similarly the projected gradient norm in \cref{lem:grad-estim} is now estimated within each species, and the symmetry of the Gaussian disorder is by a product of $r$ orthogonal groups.

\subsection{Ground State Energy on $\ell^q$ Balls}

A recent paper \cite{chen2023ell} considered the ground state of $H_N$ with Gaussian disorder in case of quadratic $H_N$ on an $\ell^q$ ball 
\[
    \cB_{N,q}=\{x\in\bbR^N~:~\|x\|_q^q=N\}.
\]
The behavior of the near-optima was shown to depend strongly on $q$: for $2\leq q\leq \infty$ the ground state is delocalized, while for $1<q<2$ it is localized on coordinates with unusually large entries.
In \cite[Open Problem 5]{chen2023ell}, it was asked whether universality holds for random matrices $\vJ^{(2)}$ with IID subgaussian disorder in the delocalized regime. 
Our methods resolve this question for $q\in (2,\infty)$.
Under the same moment conditions as our main results, we obtain universality for the $\ell^q$ ground state in mixed $p$-spin models without external field; the latter condition is in some sense necessary for $q\neq 2$ as explained in \cref{rem:p-j=1}.
Given that \cite{Cha05} showed only 2nd moment bounds are needed when $q=\infty$, we expect the optimal moment conditions to vary continuously with $q$. We did not attempt refinements in this direction.

\begin{theorem}
\label{thm:l-q-ground-state}
    Fix $q\in (2,\infty)$. Let $H_N(\cdot,\vJ)$ be as in \eqref{eq:HN-def} with $\gamma_1=0$, where condition \eqref{eq:lee-yin} holds for $\vJ$, and let $H_N(\cdot,\vG)$ have independent Gaussian disorder. 
    Then
    \[
    \plim_{N\to\infty}
    \frac{1}{N}
    \Big|
    \GS(\vJ;\cB_{N,q})
    -
    \GS(\vG;\cB_{N,q})
    \Big|
    =
    0.
    \]
    Here $\GS(\vJ;\cB_{N,q})=\max_{x\in\cB_{N,q}}H_N(x;\vJ)$ and similarly for $\vG$.
\end{theorem}

\cref{thm:l-q-ground-state} follows immediately from the next two lemmas.
Similarly to the main arguments, we show universality of free energy and then take $\beta$ large to approximate the ground state.
Thanks to the work we have already done, their proofs are now straightforward.
We let $\mu_{N,q}$ be the uniform probability measure on $\cB_{N,q}$.
We will use the fact that $[-1,1]^N=\cB_{N,\infty}\subseteq\cB_{N,q}\subseteq\cB_{N,2}$ and that all have $N$-dimensional volume $e^{\Theta(N)}$.

\begin{lemma}
\label{lem:q-free-energy-approx}
    Fix $q\in (2,\infty)$.
    For $\vJ$ as in \cref{thm:l-q-ground-state} and any $\beta\in [0,\infty)$ and $\mu\in\cP(\cB_{N,q})$:
    \[
    \plim_{N\to\infty}
    \frac{1}{N}
    \Big|
    F_{\beta}(\vJ;\mu)
    -
    F_{\beta}(\vG;\mu)
    \Big|
    =
    0.
    \]
\end{lemma}

\begin{proof}
    First since $\cB_{N,q}\subseteq\cB_{N,2}$, we may apply the same truncation as in the main argument, and thus assume the properties \eqref{eq:truncated-disorder-properties}.
    Choose $\eps>0$ small depending on $q$ and consider the function $\phi_{N,\eps}(x)=\min(\max(x,-N^{\eps}),N^{\eps})$. 
    Let $\mu_*$ be the push-forward of $\mu$ under the entry-wise mapping $x=(x_1,\dots,x_N)\mapsto \vec\phi_{N,\eps}(x)=\big(\phi_{N,\eps}(x_1),\dots,\phi_{N,\eps}(x_N)\big)$.
    \cref{prop:banach} and the universality of $\GS(H_N)$ on the $\ell^2$ sphere for pure models together imply that with probability $1-o(1)$, 
    \begin{equation}
    \label{eq:deriv-bound}
    \max_{\|z\|\in\cB_{N,2}} \|\nabla H_N(z;\vJ)\|\leq C\sqrt{N}
    \end{equation}
    for $C$ depending only on the disorder distributions $\nu_p$.
    Note that $\|x-\vec\phi_{N,\eps}(x)\|_2^2\leq \sum_{|x_i|\ge N^{\eps}}|x_i|^2\le N\cdot N^{-\eps\cdot (q-2)}$ for all $x\in\cB_N$. Hence on the event \eqref{eq:deriv-bound}, 
    \[
    F_{\beta}(\vJ;\mu)-F_{\beta}(\vJ;\mu_*)\leq N^{1-\eps\cdot (q-2)/2}.
    \]
    The same holds for $\vG$. 
    Further since we assume $\gamma_1=0$, interpolation as in \cref{lem:interpolation-on-sphere} gives
    \[
    \bbE[F_{\beta}(\vJ;\mu_*)-F_{\beta}(\vG;\mu_*)]\leq N^{1-\eps}.
    \]
    Finally \cref{lem:concen-bound} shows concentration of both $F_{\beta}(\vG;\mu_*)$ and to $F_{\beta}(\vJ;\mu_*)$, the latter since we assume \eqref{eq:truncated-disorder-properties} applies.
    Combining completes the proof.
\end{proof}

\begin{lemma}
\label{lem:q-GS-approx}
    For $\vJ$ as in \cref{thm:l-q-ground-state}, $q\in (2,\infty)$, and any $\eps>0$,
    \[
    \limsup_{N\to\infty}
    \limsup_{\beta\to\infty}
    \bbP\big[
    \big|
    F_{\beta}(\vJ;\mu_{N,q})
    -
    \GS(\vJ;\cB_{N,q})
    \big|
    \geq \eps N\big]
    =
    0.
    \]
\end{lemma}

\begin{proof}
    Clearly $F_{\beta}(\vJ;\mu_{N,q})\leq \GS(\vJ;\cB_{N,q})$.
    To show the reverse estimate, we first claim that for any $\delta>0$ there exists $K$ such that for sufficiently large $N$, given any $x\in \cB_{N,q}$ the set 
    \[
    S_{q,x,\delta}=\{y\in\cB_{N,q}~:~\|x-y\|_2\leq \delta \sqrt{N}\}
    \]
    has $N$-dimensional volume at least $e^{-KN}$.
    Indeed because $[-1,1]^N\subseteq \cB_{N,q}$, it follows that $S_{q,x,\delta}$ contains the $\ell^{\infty}$ ball of radius $\delta$ centered at $(1-\delta)x$.
    On the high-probability event that \eqref{eq:deriv-bound} holds,
    \[
    F_{\beta}(\vJ;\mu_{N,q})/N
    \geq 
    \frac{1}{\beta N}
    \log
    \int_{\bsig\in S_{q,x,\delta}}
    e^{\beta H_N(\bsig)}
    \de\bsig
    \geq 
    H_N(x)
    -
    C\delta 
    -
    \frac{K}{\beta}.
    \]
    Taking small $\delta$, arbitrary $K(\delta)$, large enough $\beta$, and arbitrary $x$ concludes the proof.
\end{proof}

\subsection{Universality of Tensor PCA}

Here we study tensor PCA, focusing for convenience on the pure case.
For $\vJ^{(1)},\vJ^{(p)}$ as before, and with $\lambda\geq 0$ a signal strength, the Hamiltonian in question is 
\begin{equation}
\label{eq:tensor-PCA-hamiltonian}
    \wh H_N(\bsig)
    =
    H_N^{(p)}(\bsig)
    +
    \lambda N (\la \bsig,\vJ^{(1)}\ra/N)^p,
\end{equation}
corresponding to the spiked tensor $\wh J^{(p)}=J^{(p)}+\lambda (J^{(1)})^{\otimes p}$.
We let $\wh G^{(p)}$ be the corresponding spiked tensor for Gaussian disorder $\vG^{(1)},\vG^{(p)}$.

\begin{theorem}
\label{thm:tensor-PCA}
Assume $\vJ^{(1)},\vJ^{(p)}$ obey $(C,\eps)$-moment bounds. 
Then for $\beta\in [0,\infty]$, 
\[
\plim_{N\to\infty}
|F_{\beta}(\wh J^{(p)})-F_{\beta}(\wh G^{(p)})|/N = 0.
\]
\end{theorem}

\begin{proof}
    As usual, we may assume the disorder has been truncated. For convenience we assume $p$ is even so the function $x\mapsto x^p$ is convex. 
    The case of odd $p$ can be handled using convexity of $x\mapsto |x|^p$. 
    Indeed replacing the last term in \eqref{eq:tensor-PCA-hamiltonian} by $\lambda N |\la \bsig,\vJ^{(1)}\ra/N|^p$ gives asymptotically the same free energy thanks to \cref{lem:concen-bound} and the reflection-symmetry of $\cS_N$ along the hyperplane orthogonal to $\vJ^{(1)}$.

    The point is to read off the tensor PCA free energy from models with covariance $\xi_h(t)=t^p+h^2 t$ for varying $h\in\bbR$, using convex duality.
    For $\alpha$ tending to $0$ slowly with $N$ with $1/\alpha\in\bbZ_+$, we let $h$ range over $\alpha\bbZ\cap [-1/\alpha,1/\alpha]$. 
    Theorem~\ref{thm:main} implies that for each such $h$:
    \[
    \plim_{N\to\infty}
    |F_{\beta}(\vJ^{(p)},h\vJ^{(1)})-F_{\beta}(\vG^{(p)},h\vG^{(1)})|/N=0.
    \]
    Using the convexity of $x\mapsto x^p$ and the gradient estimate \eqref{eq:gradient-bound}, the desired result easily follows.
    In particular one uses that 
    \[
    x^p 
    = 
    \max_{z\in\bbR}
    \Big\{
    (x-z)pz^{p-1}+z^p-z
    \Big\}
    .
    \]
    Taking $h=pz^{p-1}$ and using \eqref{eq:gradient-bound} gives an approximate representation of $F_{\beta}(\wh J^{(p)})$ as a maximum of finitely many shifts of ordinary free energies $F_{\beta}(\vJ^{(p)},h\vJ^{(1)})$.
    Since these are universal, the proof is completed.
\end{proof}

We note that the limiting free energy $\plim_{N\to\infty} F_{\beta}(\wh G^{(p)})/N$ can easily be read off from the Parisi formula as a maximum over bands with fixed correlation with $\vG^{(1)}$.

In the case of Gaussian disorder, the above model for tensor PCA was introduced in \cite{richard2014statistical} as a canonical example of a high-dimensional signal detection problem.
One for instance observes either $\wh H_N$ above or an ordinary pure $p$-spin Hamiltonian (i.e. with $\lambda=0$) and aims to distinguish between the two cases.
It is known that the log-likelihood ratio between the null and planted tensor PCA models is given by the normalized free energy of the observed Hamiltonian.
Building on this connection, \cite{jagannath2020statistical} proved that the threshold $\lambda_c$ above which the null and planted distribution are statistically distinguishable coincides with the replica-symmetry breaking threshold for the pure spherical $p$-spin model.
\cref{thm:tensor-PCA} shows $1$-sided universality (in both noise and signal) of the statistical distinguishability threshold, in the sense that simply computing the free energy continues to serve as a distinguishing test.
Of course, the opposite side cannot be universal in general, since the tensor entries could encode microscopic information to make detection trivial for any $\lambda>0$.

\section*{Acknowledgements}

Thanks to Reza Gheissari and Patrick Lopatto for helpful comments and discussions.

A portion of this research was conducted during the period Sawhney served as a Clay Research Fellow. Sawhney was also supported by NSF Graduate Research Fellowship Program DGE-1745302.

\bibliographystyle{amsplain0}
\bibliography{main.bib}

\providecommand{\bysame}{\leavevmode\hbox to3em{\hrulefill}\thinspace}
\providecommand{\MR}{\relax\ifhmode\unskip\space\fi MR }
\providecommand{\MRhref}[2]{%
  \href{http://www.ams.org/mathscinet-getitem?mr=#1}{#2}
}
\providecommand{\href}[2]{#2}
\begin{thebibliography}{10}

\bibitem{auffinger2015parisi}
Antonio Auffinger and Wei-Kuo Chen, \emph{The {P}arisi formula has a unique minimizer}, Comm. Math. Phys. \textbf{335} (2015), 1429--1444.

\bibitem{auffinger2016universality}
Antonio Auffinger and Wei-Kuo Chen, \emph{Universality of chaos and ultrametricity in mixed {$p$}-spin models}, Comm. Pure Appl. Math. \textbf{69} (2016), 2107--2130.

\bibitem{BY88}
Z.~D. Bai and Y.~Q. Yin, \emph{Necessary and sufficient conditions for almost sure convergence of the largest eigenvalue of a {W}igner matrix}, Ann. Probab. \textbf{16} (1988), 1729--1741.

\bibitem{Ban38}
Stefan Banach, \emph{{{\"U}ber homogene Polynome in $L^2$}}, Studia Mathematica \textbf{7} (1938), 36--44.

\bibitem{barra2015multi}
Adriano Barra, Pierluigi Contucci, Emanuele Mingione, and Daniele Tantari, \emph{Multi-species mean field spin glasses. {R}igorous results}, Ann. Henri Poincar\'{e} \textbf{16} (2015), 691--708.

\bibitem{CH06}
Philippe Carmona and Yueyun Hu, \emph{Universality in {S}herrington-{K}irkpatrick's spin glass model}, Ann. Inst. H. Poincar\'{e} Probab. Statist. \textbf{42} (2006), 215--222.

\bibitem{Cha05}
Sourav Chatterjee, \emph{A simple invariance theorem}, arXiv:math/0508213.

\bibitem{Che13}
Wei-Kuo Chen, \emph{The {A}izenman-{S}ims-{S}tarr scheme and {P}arisi formula for mixed {$p$}-spin spherical models}, Electron. J. Probab. \textbf{18} (2013), no. 94, 14.

\bibitem{chen2023some}
Wei-Kuo Chen, Heejune Kim, and Arnab Sen, \emph{{Some Rigorous Results on the L{\'e}vy Spin Glass Model}}.

\bibitem{chen2024universality}
Wei-Kuo Chen and Wai-Kit Lam, \emph{{Universality of superconcentration in the Sherrington--Kirkpatrick model}}, Random Structures \& Algorithms \textbf{64} (2024), 267--286.

\bibitem{CS17}
Wei-Kuo Chen and Arnab Sen, \emph{Parisi formula, disorder chaos and fluctuation for the ground state energy in the spherical mixed {$p$}-spin models}, Comm. Math. Phys. \textbf{350} (2017), 129--173.

\bibitem{chen2023ell}
Wei-Kuo Chen and Arnab Sen, \emph{On {$\ell_p$}-{G}aussian-{G}rothendieck problem}, Int. Math. Res. Not. IMRN (2023), 2344--2428.

\bibitem{Coo12}
Nick Cook, \emph{Notes on {T}alagrand's {I}soperimetric {I}nequality}, 2012.

\bibitem{crisanti1992spherical}
Andrea Crisanti and H-J Sommers, \emph{The spherical p-spin interaction spin glass model: the statics}, Zeitschrift f{\"u}r Physik B Condensed Matter \textbf{87} (1992), 341--354.

\bibitem{guerra2003broken}
Francesco Guerra, \emph{Broken replica symmetry bounds in the mean field spin glass model}, Comm. Math. Phys. \textbf{233} (2003), 1--12.

\bibitem{huang2023constructive}
Brice Huang and Mark Sellke, \emph{{A Constructive Proof of the Spherical Parisi Formula}}.

\bibitem{jagannath2022existence}
Aukosh Jagannath and Patrick Lopatto, \emph{Existence of the free energy for heavy-tailed spin glasses}.

\bibitem{jagannath2020statistical}
Aukosh Jagannath, Patrick Lopatto, and Leo Miolane, \emph{{Statistical thresholds for tensor PCA}}, The Annals of Applied Probability \textbf{30} (2020), 1910--1933.

\bibitem{LT11}
Michel Ledoux and Michel Talagrand, \emph{Probability in {B}anach spaces}, Classics in Mathematics, Springer-Verlag, Berlin, 2011, Isoperimetry and processes, Reprint of the 1991 edition.

\bibitem{lee2014necessary}
Ji~Oon Lee and Jun Yin, \emph{{A necessary and sufficient condition for edge universality of Wigner matrices}}, Duke Mathematical Journal \textbf{163} (2014), 117--173.

\bibitem{panchenko2013parisi}
Dmitry Panchenko, \emph{{The Parisi ultrametricity conjecture}}, Ann. Math. (2013), 383--393.

\bibitem{panchenko2015free}
Dmitry Panchenko, \emph{The free energy in a multi-species {S}herrington-{K}irkpatrick model}, Ann. Probab. \textbf{43} (2015), 3494--3513.

\bibitem{parisi1979infinite}
Giorgio Parisi, \emph{Infinite number of order parameters for spin-glasses}, Phys. Rev. Lett. \textbf{43} (1979), 1754.

\bibitem{richard2014statistical}
Emile Richard and Andrea Montanari, \emph{{A statistical model for tensor PCA}}, Advances in neural information processing systems \textbf{27} (2014).

\bibitem{subag2021tap1}
Eliran Subag, \emph{{TAP approach for multi-species spherical spin glasses I: general theory}}.

\bibitem{subag2021tap2}
Eliran Subag, \emph{T{AP} approach for multispecies spherical spin glasses {II}: the free energy of the pure models}, Ann. Probab. \textbf{51} (2023), 1004--1024.

\bibitem{Tal02}
Michel Talagrand, \emph{Gaussian averages, {B}ernoulli averages, and {G}ibbs' measures}, vol.~21, 2002, Random structures and algorithms (Poznan, 2001), pp.~197--204.

\bibitem{Tal06}
Michel Talagrand, \emph{Free energy of the spherical mean field model}, Probab. Theory Related Fields \textbf{134} (2006), 339--382.

\bibitem{talagrand2006parisi}
Michel Talagrand, \emph{The {P}arisi formula}, Ann. of Math. (2) \textbf{163} (2006), 221--263.

\end{thebibliography}
\end{document}